\documentclass[11pt, a4paper, oneside]{amsart}
    \usepackage[english]{babel}
    \usepackage{amssymb,enumerate,bbm,xcolor}
    \usepackage[a4paper, total={6in, 8in}]{geometry}
    \usepackage{stix}
    
    \catcode`\<=\active \def<{
    \fontencoding{T1}\selectfont\symbol{60}\fontencoding{\encodingdefault}}
    \catcode`\>=\active \def>{
    \fontencoding{T1}\selectfont\symbol{62}\fontencoding{\encodingdefault}}
    \newcommand{\assign}{:=}
    \newcommand{\equallim}{\mathop{=}\limits}
    \newcommand{\mathd}{\mathrm{d}}
    \newcommand{\mathe}{\mathrm{e}}
    \newcommand{\nobracket}{}
    
    \newcommand{\tmcolor}[2]{{\color{#1}{#2}}}
    \newcommand{\tmdfn}[1]{\textbf{#1}}
    \newcommand{\tmem}[1]{{\em #1\/}}
    \newcommand{\tmfoldedenv}[2]{\trivlist{\item[]\mbox{}#1}}
    \newcommand{\tmop}[1]{\ensuremath{\operatorname{#1}}}
    \newcommand{\tmtextit}[1]{{\itshape{#1}}}
    \newenvironment{enumerateroman}{\begin{enumerate}[i.] }{\end{enumerate}}
    
    \newtheorem{corollary}{Corollary}
    \newtheorem{lemma}{Lemma}
    \newtheorem{proposition}{Proposition}
    \newtheorem{remark}{Remark}
    \newtheorem{theorem}{Theorem}
    
\begin{document}
    
    \title{ Microscopic derivation of the Keller-Segel equation in the
    sub-critical regime}
    
    \author{Ana Ca{\~n}izares Garc{\'i}a}
    \address{LMU-Munich}
    \email{caniz@math.lmu.de}
    
    \author{Peter Pickl}
    \address{LMU-Munich}
    \email{pickl@math.lmu.de}
    
    \date{February, 6th 2017}
    
    \begin{abstract}
      We derive the two-dimensional Keller-Segel equation from a stochastic system
      of $N$ interacting particles in the case of sub-critical chemosensitivity
      $\chi < 8 \pi$. The Coulomb interaction force is regularised with a cutoff
      of size $N^{- \alpha}$, with arbitrary $\alpha \in (0, 1 / 2)$. In
      particular we obtain a quantitative result for the maximal distance between
      the real and mean-field $N$-particle trajectories.
    \end{abstract}
    
    {\maketitle}
    
    \tmfoldedenv{\ }{\subsubsection{TO-DO: Intro}
    
    \subsubsection{TO-DO: Main result}
    
    -> Determine $N_0$
    
    - Comment on empirical measures
    
    \subsubsection{TO-DO: Properties of solutions}
    
    Proof of the $L^{\infty}$- estimates for $\rho^N$ and deduce the ones for
    $\rho$ from here.
    
    \subsubsection{TO-DO: Lipschitz bound}
    
    \subsubsection{TO-DO: LoLN}
    
    \subsubsection{TO-DO: Lemma}
    
    -Existence, uniqueness of the regularised and linearised macroscopic equation
    (\ref{macroscopic-eq-linearised}) with $\delta$ initial condition. Processes
    with the strong Feller property have a transition probability density, which
    solves the linearised KS equation. Does $Z^x$ have the strong Feller property?
    
    -Variation of constants (\ref{variation-of-constants}) for $u_t^{x_i}$
    
    -$\| G (t, x - x_0) - G (t, x - y_0) \|_1$}
    
    \section{Introduction}\label{equations}
    
    The Keller-Segel equation {\cite{keller_initiation_1970}} is known as the
    classical model of {\tmem{chemotaxis}}, which in Biology refers to the
    movement of organisms guided by an external chemical substance and has been
    observed in some species of bacteria or amoeba. The Keller-Segel equation,
    concretely motivated by the behaviour of the unicellular organism
    {\tmem{Dictyostelium discoidium}}, models a situation in which cells naturally
    spread out but under starvation circumstances also attract other cells by
    segregating an attractive chemical substance. We consider the two-dimensional
    Keller-Segel equation:
    \begin{equation}
      \partial_t \rho = \Delta \rho + \chi \nabla ((k \ast \rho) \rho), \qquad
      \rho (0, \cdot) = \rho_0 \label{macroscopic-eq} .
    \end{equation}
    Here $\rho : [0, \infty) \times \mathbbm{R}^2 \rightarrow [0, \infty)$ is the
    evolution of the cell population density for an initial value $\rho_0 :
    \mathbbm{R}^2 \rightarrow [0, \infty)$, the interaction force kernel $k :
    \mathbbm{R}^2 \rightarrow \mathbbm{R}^2$ is given by $k (x) \assign \frac{x}{2
    \pi | x |^2}$ and the constant $\chi > 0$ denotes the
    {\tmem{chemosensitivity}} or response of the cells to the chemical
    substance\footnote{This form of the equation results from the Keller-Segel
    parabolic-elliptic system if the concentration of chemical substance is taken
    to be the Newtonian potential of the density of cells
    {\cite{haskovec_stochastic_2009}}.}. This model reflects the characteristic
    competition between diffusion and aggregation in such a chemotactical process.
    Mathematically this results in the interesting effect that in some cases
    smooth solutions exist for all times, while in others solutions {\tmem{blow
    up}} in finite time\footnote{A solution $\rho (t, x)$ is said to blow up in
    finite time if $\lim_{t \rightarrow T}  \| \rho (t, \cdot) \|_{L^{\infty}} =
    \infty$ for some finite time $T$.} (corresponding to clustering of the cells).
    Furthermore, the existence of global solutions or the presence of blow-up
    events strongly depend on the dimension, mass and chemosensitivity of the
    system: in one dimension the solution exists globally, but in higher
    dimensions blow-up events in finite time may or may not occur depending on the
    initial mass $M \assign \int_{\mathbbm{R}^2} \rho_0 (x)$ and the
    chemosensitivity $\chi$. This role for the $2$-dimensional description was
    completely understood for the first time less than a decade ago: if $\chi M <
    8 \pi$, a global bounded solution exists, while for $\chi M > 8 \pi$ blow-up
    in finite time always takes place. Finally, if $\chi M = 8 \pi$ a global
    solution exists which possibly becomes unbounded as $t \rightarrow \infty$
    {\cite{blanchet_two-dimensional_2006}}, {\cite{dolbeault_optimal_2004}},
    {\cite{blanchet_infinite_2008}}. Here we work in a probabilistic setting and
    for convenience assume an initial mass $M = 1$. The threshold condition for
    the existence of global solutions is therefore at $\chi = 8 \pi$.
    
    Our purpose in this paper is to derive the deterministic {\tmem{macroscopic}}
    equation (\ref{macroscopic-eq}) in the sub-critical regime $\chi \in (0, 8
    \pi)$ as the mean-field limit of the following {\tmem{microscopic}} stochastic
    $N$-particle system as $N \rightarrow \infty$:
    \begin{equation}
      \mathd X_t^{i (N)} = - \frac{\chi}{N}  \sum^N_{j \neq i} k (X_t^{i (N)} -
      X_t^{j (N)}) \mathd t + \sqrt{2} \mathd B^i_t, \quad i = 1, \ldots, N, \quad
      X^{(N)}_0 \sim \bigotimes_{i = 1}^N \rho_0, \label{particle-system}
    \end{equation}
    where the process $X^{i (N)} : [0, \infty) \rightarrow \mathbbm{R}^2$ denotes
    the trajectory of the $i$-th particle, $(B^i)_{i \in \mathbbm{N}}$ is a family
    of 2-dimensional independent Brownian motions, $X^{(N)}_t \in \mathbbm{R}^{2
    N}$ denotes the vector $X^{(N)}_t \assign (X_t^{1 (N)}, \ldots, X_t^{N
    (N)})$\footnote{We introduce the notation $(N)$ for the number of particles in
    order to differentiate between these trajectories and the regularised ones. We
    nevertheless just use this notation during the introduction, since in the
    following sections we only work with the regularised equations.}, and at the
    initial time $t = 0$ the particles are independently distributed according to
    the initial density $\rho_0$. To this end we prove the property of
    {\tmdfn{propagation of chaos}} for regularised versions \ (with a cutoff
    depending on $N$) of these equations in Corollary \ref{main-corollary}. We
    obtain the propagation of chaos as a consequence of Theorem \ref{main-result},
    where the real trajectories $X^{i (N)}$ are shown to remain close to the
    mean-field trajectories, defined by (\ref{microscopic-mean}) below, if both
    started at the same point. The mean-field trajectories are given by the
    following equation:
    \begin{equation}
      \mathd Y_t^{i (N)} = - \chi  (k \ast \rho_t) (Y_t^{i (N)}) \mathd t +
      \sqrt{2} \mathd B_t^i, \quad i = 1, \ldots, N, \quad Y^{(N)}_0 = X^{(N)}_0,
      \label{microscopic-mean}
    \end{equation}
    where $\rho_t =\mathcal{L} (Y^{i (N)}_t)$ is the probability distribution of
    any of the i.i.d. $Y_t^{i (N)}$. We remark that the Keller-Segel equation
    (\ref{macroscopic-eq}) is Kolmogorov's forward equation for any solution of
    (\ref{microscopic-mean}), and in particular their probability distribution
    $\rho_t$ solves (\ref{macroscopic-eq}).
    
    The work of Cattiaux and P{\'e}d{\`e}ches {\cite{cattiaux_2-d_2016}} is
    relevant for the existence of solutions of the stochastic particle system
    (\ref{particle-system}) and their properties. Furthermore, the derivation of
    the macroscopic equation (\ref{macroscopic-eq}) from the many-particle system
    (\ref{particle-system}) or propagation of chaos has been addressed in the past
    years by several mathematicians for modified problems: for a regularised
    interaction force $k_{\varepsilon} (x) \assign \frac{x}{| x |  (| x | +
    \varepsilon)}$ in {\cite{haskovec_stochastic_2009}} and for a sub-Keller-Segel
    equation with a less singular force $k_{\alpha} (x) \assign \frac{x}{| x
    |^{\alpha + 1}}$, $0 < \alpha < 1$, in {\cite{godinho_propagation_2013}}. More
    recently, great progress has been made for the purely Coulomb case ($\alpha =
    1$): Fournier and Jourdain {\cite{fournier_stochastic_2015}} proved the
    convergence of a subsequence for the particle system (\ref{particle-system})
    by a tightness argument in the {\tmem{very sub-critical}} case $\chi < 2 \pi$
    using no cutoff at all; the convergence of the whole sequence (and therefore
    propagation of chaos) was nevertheless not achieved. Liu et al. published in
    the past year several results on propagation of chaos of (\ref{particle-system}) {\cite{liu_random_2015}}, {\cite{huang_error_2016}},
    {\cite{liu_convergence_2015-1}}, the last of them containing the strongest
    result available to date to our knowledge. We improve their result in two
    aspects. On the one hand our conditions (\ref{condition-initial-data}) on the
    initial density $\rho_0$ are weaker: Liu and Zhang assume $\rho_0$ is
    compactly supported, Lipschitz continuous and $\rho_0 \in H^4
    (\mathbbm{R}^2)$. On the other hand our initial configuration for the $N$
    particles is less restrictive: ours are i.i.d. random variables on
    $\mathbbm{R}^2$, while their particles are distributed on a grid. Our approach
    adapts a method that seems to be powerful for deriving the mean-field limit of
    some $N$-particle systems with Coulomb interactions, which was presented by
    Boers, Pickl {\cite{boers_mean_2016}} and Lazarovizi, Pickl
    {\cite{lazarovici_mean-field_2015}} for the derivation of the Vlasov-Poisson
    equation from an $N$-particle Coulomb system for typical initial
    conditions.
    
    \
    
    \paragraph{Conditions on the chemosensitivity and the initial density}
    
    We assume throughout this note a sub-critical chemosensitivity $\chi \in (0, 8
    \pi)$ and the following conditions on the initial density $\rho_0$:
    \begin{eqnarray}
      \rho_0 & \in & L^1 (\mathbbm{R}^2, (1 + | x |^2) \mathd x) \cap L^{\infty}
      (\mathbbm{R}^2) \cap H^2 (\mathbbm{R}^2), \nonumber\\
      \rho_0 & \geqslant & 0, \nonumber\\
      \int_{\mathbbm{R}^2} \rho_0 (x) \mathd x & = & 1, \nonumber\\
      \rho_0 \log \rho_0 & \in & L^1 (\mathbbm{R}^2) .  \label{condition-initial-data}
    \end{eqnarray}
    These conditions guarantee global existence, uniqueness and further good
    properties of the solution of the macroscopic equation (\ref{macroscopic-eq}).
    Section \ref{properties-of-solutions} reviews these results and the
    corresponding ones for the solutions of the microscopic system.
    
    \paragraph{Regularisation of the interaction force}{\nopagebreak}
    
    We introduce the following $N$-dependent regularisation of the Coulomb
    interaction force. Let $\phi^1 : \mathbbm{R}^2 \rightarrow [0, \infty)$ be a
    radially symmetric, smooth function with the following properties:
    \[ \phi^1 (x) \assign \left\{ \begin{array}{ll}
         - \frac{1}{2 \pi} \log | x |, & | x | \geqslant 2,\\
         0, & | x | \leqslant 1,
       \end{array} \right. \]
    $| \nabla \phi^1 (x) | \leqslant (2 \pi | x |)^{- 1}$, $- \Delta \phi^1 (x)
    \geqslant 0$ and $| \partial^2_{i j} \phi^1 (x) | \leqslant (\pi | x |^2)^{-
    1}$ for any $x \in \mathbbm{R}^2$, $i, j = 1, 2$. For each $N \in \mathbbm{N}$
    and $\alpha \in (0, 1 / 2)$, let $\phi^N (x) = \phi^1 (N^{\alpha} x)$ and
    consider the regularised interaction force $k^N = - \nabla \phi^N$, which by
    construction satisfies
    \[ k^N (x) \assign \left\{ \begin{array}{ll}
         \frac{x}{2 \pi | x |^2}, & | x | \geqslant 2 N^{- \alpha}\\
         0, & | x | \leqslant N^{- \alpha}
       \end{array} \right. \]
    and
    \[ | \partial_i k^N (x) | \leqslant \left\{ \begin{array}{ll}
         \frac{1}{\pi | x |^2}, & | x | > N^{- \alpha}\\
         0, & | x | \leqslant N^{- \alpha}
       \end{array} \right., \quad i = 1, 2. \]

    For an initial density $\rho_0$ satisfying the above conditions
    (\ref{condition-initial-data}) and each $N \in \mathbbm{N}$ we consider the
    regularised Keller-Segel equation
    \begin{equation}
      \partial_t \rho^N = \Delta \rho^N + \chi \nabla ((k^N \ast \rho^N) \rho^N),
      \qquad \rho^N (0, \cdot) = \rho_0 \label{macroscopic-eq-regularised},
    \end{equation}
    the regularised microscopic $N$-particle system, for $i = 1, \ldots, N$,
    \begin{equation}
      \mathd X_t^{i (N), N} = - \frac{\chi}{N}  \sum_{j \neq i} k^N (X_t^{i (N),
      N} - X_t^{j (N), N}) \mathd t + \sqrt{2} \mathd B^i_t, \quad i = 1, \ldots,
      N, \quad X_0^{(N), N} \sim \bigotimes_{i = 1}^N \rho_0, \label{microscopic-eq-regularised}
    \end{equation}
    and the regularised mean-trajectories
    \begin{equation}
      \mathd Y_t^{i (N), N} = - \chi  (k^N \ast \rho^N_t) (Y_t^{i (N), N})
      \mathd t + \sqrt{2} \mathd B_t^i, \quad i = 1, \ldots, N, \quad Y^{(N), N}_0
      = X^{(N), N}_0 \label{microscopic-mean-regularised}
    \end{equation}
    where $\rho^N_t$ denotes the probability distribution of $Y^{i (N), N}_t$, for
    any $i = 1, \ldots, N$. As in the non-regularised version this implies that
    $\rho^N$ solves the regularised Keller-Segel equation (\ref{macroscopic-eq-regularised}).  For simplicity of notation, and since the number of particles
    $N$ already becomes apparent by the dependency of $N$ of the cutoff, we will
    just write $X^{i, N}$ and $Y^{i, N}$ instead of $X^{i (N), N}$ and $Y^{i (N),
    N}$, as well as $X^N$ and $Y^N$ for the vectors $X^{(N), N}$ and $Y^{(N), N}$.
    It is also convenient to denote the regularised interaction force as
    \begin{equation}
      K^N_i (x_1, \ldots, x_N) \assign - \frac{\chi}{N}  \sum_{j \neq i} k^N (x_i
      - x_j) \label{regularised-force}
    \end{equation}
    and the mean interaction force as
    \[ \overline{K}^N_{t, i} (x_1, \ldots, x_N) \assign - \chi (k^N \ast \rho^N_t)
       (x_i), \]
    where $\rho^N_t =\mathcal{L} (Y_t^{i, N})$. We need to introduce one last
    process: For times $0 \leqslant s \leqslant t$ and a random variable $X \in
    \mathbbm{R}^{2 N}$, independent of the filtration generated by $B_r$, $r
    \geqslant s$, let $Z^{X, N}_{t, s}$ be the process starting at time $s$ at the
    position $X$ and evolving from time $s$ up to time $t$ with the mean force
    $\overline{K}^N$. Put in another way, the process $Z^{X, N}_{t, s} = (Z^{X, 1,
    N}_{t, s}, \ldots, Z^{X, N, N}_{t, s})$ is given by the solution of
    \begin{equation}
      \mathd Z_{t, s}^{X, i, N} = \overline{K}^N_{t, i} (Z_{t, s}^{X, N}) \mathd t
      + \sqrt{2} \mathd B_t^i, \quad i = 1, \ldots, N, \quad Z^{X, N}_{s, s} = X.
      \label{microscopic-mean-mixed}
    \end{equation}
    This paper is organised as follows. In the next section we state our main
    result and the ensuing propagation of chaos. We comment on the existence and
    properties of solutions of equations (\ref{macroscopic-eq})-(\ref{microscopic-mean-mixed}) in Section \ref{properties-of-solutions}. Section
    \ref{preliminary-results} is devoted to some preliminary results that we need
    for the proof of the main result, Theorem \ref{main-result}, which is then
    proven in Section \ref{proof-main-theorem}. We conclude with the proofs of
    Propositions \ref{Linf-estimates} and \ref{hˆlder-estimates} introduced in
    Section \ref{properties-of-solutions}.
    
    \paragraph{Notation}For simplicity we write single bars $| \cdot |$ for norms
    in $\mathbbm{R}^n$ and $\| \cdot \|$ for norms in $L^p$ spaces.
    
    \section{Main result}
    
    Let the chemosensitivity $\chi$ and the initial density $\rho_0$ satisfy
    condition (\ref{condition-initial-data}), and for $N \in \mathbbm{N}$ let
    $X^N$ and $Y^N$ be the real and mean-field trajectories solving the
    regularised microscopic equations (\ref{microscopic-eq-regularised}) and
    (\ref{microscopic-mean-regularised}), respectively. Our main result is that
    the $N$-particle trajectory $X^N$ starting from a chaotic
    (product-distributed) initial condition $X^N_0 \sim \otimes^N_{i = 1} \rho_0$
    typically remains close to the purely chaotic mean-field trajectory $Y^N$ with
    same initial configuration $Y_0^N = X_0^N$ during any finite time interval
    $[0, T]$. More precisely, we prove that the measure of the set where the
    maximal distance $| X^N_t - Y^N_t |_{\infty}$ on $[0, T]$ excedes $N^{-
    \alpha}$ decreases exponentially with the number of particles $N$, as the
    number of particles grows to infinity.
    
    \begin{theorem}
      \label{main-result}Let $T > 0$ and $\alpha \in (0, 1 / 2)$. For each $\gamma
      > 0$, there exist a positive constant $C_{\gamma}$ and a natural number
      $N_0$ such that
      \[ \mathbbm{P} (\sup_{0 \leqslant t \leqslant T} | X^N_t - Y^N_t |_{\infty}
         \geqslant N^{- \alpha}) \leqslant C_{\gamma} N^{- \gamma}, \quad
         \text{for each } N \geqslant N_0 . \]
      The constant $C_{\gamma}$ depends on the coefficient $\chi$, the initial
      density $\rho_0$, the final time $T$, $\alpha$ and $\gamma$ and $N_0$
      depends on $\rho_0$, $T$ and $\alpha$.
    \end{theorem}
    
    Note that if the interaction force were Lipschitz continuous the statement
    would easily follow from a Gr{\"o}nwall-type argument. In our case we do not
    have this good property, but we can prove that the regularised force $K^N$ is
    locally Lipschitz with a bound of order $\log N$, which follows from Lemma
    \ref{lipschitz-bound-G} and the Law of large numbers as presented in
    Proposition \ref{loln-general}. This Lipschitz bound is good enough to prove
    the statement for short times but for larger times we need to introduce a new
    intermediate process. This process is proved to be close to $X^N_t$ by the
    same argument as before for short times and close to $Y^N_t$ by a new argument
    introduced in Lemma \ref{lemma} which compares the densities of the processes
    instead of comparing the trajectories.
    
    We remark that Theorem \ref{main-result} directly implies the propagation of
    chaos, or the weak convergence of the $k$-particle marginals for $X^N_t$ and
    $Y^N_t$:
    
    \begin{corollary}
      \label{main-corollary}Consider the probability density $\otimes_{i = 1}^N
      \rho_t^N$ of $Y_t^N$ and denote by $\Psi^N_t$ the probability density of
      $X^N_t$. Then, for each $\gamma > 0$, there exist a positive constant
      $C_{\gamma}$ and a natural number $N_0$ such that
      \begin{equation}
        \sup_{0 \leqslant t \leqslant T} W_1 (^{(k)} \Psi^N_t, \otimes_{i = 1}^k
        \rho_t^N) \leqslant C_{\gamma} N^{- \gamma} \label{wasserstein-k-marginals}
      \end{equation}
      holds for each $k \in \mathbbm{N}$ and $N \geqslant N_0$. $W_1$ denotes the
      first Wassertein distance, the constant $C_{\gamma}$ depends on the
      coefficient $\chi$, the initial density $\rho_0$, the final time $T$,
      $\alpha$ and $\gamma$ and $N_0$ depends on $\rho_0$, $T$ and $\alpha$. Here
      the constant $C_{\gamma}$ might be different from the one in Theorem
      \ref{main-result}.
    \end{corollary}
    
    \tmfoldedenv{\ }{The first Wassertein distance $W_1$ of two probability
    measures $\mu, \nu$ on $(\mathbbm{R}^n, \| \cdot \|_{\infty})$ is defined as
    \[ W_1 (\mu, \nu) \assign \inf_{\pi \in \Pi (\mu, \nu)} \int_{\mathbbm{R}^n
       \times \mathbbm{R}^n} \| x - y \|_{\infty} \mathd \pi (x, y), \]
    where $\Pi (\mu, \nu)$ is the set of all probability measures on
    $\mathbbm{R}^n \times \mathbbm{R}^n$ with first marginal $\mu$ and second
    marginal $\nu$. One can analogously define for $p \in [1, \infty)$ the $p$-th
    Wasserstein distance. Convergence in any Wasserstein distance implies weak
    convergence in $\mathcal{P} (\mathbbm{R}^n)$ (the set of probability measures
    on $(\mathbbm{R}^n, \| \cdot \|_{\infty})$). We next show that the result in
    Theorem \ref{main-result} is indeed stronger than (\ref{wasserstein-k-marginals}). To this end we use the Kantonovich-Rubinstein duality for the
    first Wasserstein distance
    \[ W_1 (\mu, \nu) = \sup_{\| g \|_L \leqslant 1} \left\{ \int g (x) \mathd \mu
       (x) - \int g (x) \mathd \nu (x) \right\}, \]
    where the supremum is taken over Lipschitz functions $g : \mathbbm{R}^n
    \rightarrow \mathbbm{R}$, for which the Lipschitz norm is given by $\| g \|_L
    \assign \sup_{x, y} \frac{| g (x) - g (y) |}{\| x - y \|_{\infty}} $. Let now
    $\gamma > 0$ and assume the statement in Theorem \ref{main-result} holds.
    Then, for each $t \in [0, T]$, the Wasserstein distance of $^{(k)} \Psi^N_t$
    and $\otimes_{i = 1}^k \rho_t^N$ (probability densities on $\mathbbm{R}^{2
    k}$) is
    \begin{eqnarray*}
      W_1 (^{(k)} \Psi^N_t, \otimes_{i = 1}^k \rho_t^N) & = & \sup_{\| g \|_L
      \leqslant 1} \left\{ \int_{\mathbbm{R}^{2 k}} g (x) ^{(k)} \Psi^N_t (x)
      \mathd x - \int_{\mathbbm{R}^{2 k}} g (x) \otimes_{i = 1}^k \rho_t^N (x)
      \mathd x \right\}\\
      & = & \sup_{\| g \|_L \leqslant 1} \int \{ g (X^{1, N}_t, \ldots, X^{k,
      N}_t) - g (Y^{1, N}_t, \ldots, Y^{k, N}_t) \} \mathd \mathbbm{P}\\
      & \leqslant & \int \| (X^{1, N}_t, \ldots, X^{k, N}_t) - (Y^{1, N}_t,
      \ldots, Y^{k, N}_t) \|_{\infty} \mathd \mathbbm{P}
    \end{eqnarray*}
    for $k \in \{ 1, \ldots, N \}$. Let us define the set $A \assign \{ \sup_{0
    \leqslant t \leqslant T} \| X^N_t - Y^N_t \|_{\infty} \geqslant N^{- \alpha}
    \}$ and compute the last integral over $\Omega \backslash A$ and over $A$
    separately. On the complementary set of $A$ the integral is just
    \[ \int_{\Omega \backslash A} \| (X^{1, N}_t, \ldots, X^{k, N}_t) - (Y^{1,
       N}_t, \ldots, Y^{k, N}_t) \|_{\infty} \mathd \mathbbm{P} \leqslant N^{-
       \alpha} . \]
    On the other hand, $\int_A \| (X^{1, N}_t, \ldots, X^{k, N}_t) - (Y^{1, N}_t,
    \ldots, Y^{k, N}_t) \|_{\infty} \mathd \mathbbm{P}$ is bounded by
    \begin{eqnarray*}
      \int_A \| X^N_t - Y^N_t \|_{\infty} \mathd \mathbbm{P} & = & \int_A \left\|
      \int^t_0 K^N (X^N_s) - \overline{K}^N_s (Y^N_s) \mathd s \right\|_{\infty}
      \mathd \mathbbm{P}\\
      & \leqslant & t (\| K^N \|_{\infty} + \sup_{0 \leqslant s \leqslant t} \|
      \overline{K}^N_s \|_{\infty}) \mathbbm{P} (A)\\
      & \leqslant & T (2 \pi)^{- 1} N^{\alpha} C_{2 \alpha} N^{- 2 \alpha}\\
      & \leqslant & CN^{- \alpha},
    \end{eqnarray*}
    for a constant $C$ depending on $\chi$, $\rho_0$, $T$ and $\alpha$. We
    conclude that
    \[ W_1 (^{(k)} \Psi^N_t, \otimes_{i = 1}^k \rho_t^N) \leqslant CN^{- \alpha}
    \]
    holds for each $t \in [0, T]$ and, since the bound does not depend on $t$ but
    only on the final time $T$, the same bound holds after taking the supremum
    over $0 \leqslant t \leqslant T$.
    
    \tmcolor{red}{A result comparing $\Psi^N_t$ and $\otimes_{i = 1}^N \rho_t$
    would be much more interesting (the latter being the solution of Keller-Segel
    without cutoff). We know that $\rho_t^N$ converges weakly to $\rho_t$ by a
    compactness argument (see Proposition \ref{macro-existence-and-convergence}),
    but we don't know the rate of convergence.}}
    
    \section{Properties of solutions}\label{properties-of-solutions}
    
    \subsection{Macroscopic equations}
    
    \begin{proposition}
      \label{macro-existence-and-convergence} {\tmem{(Existence and convergence)}}
      Under assumption (\ref{condition-initial-data}) for the chemosensitivity
      $\chi$ and the initial density $\rho_0$ the following holds:
      \begin{enumerateroman}
        \item For any $N \in \mathbbm{N}$ and any $T > 0$, there exists $\rho^N
        \in L^2 (0, T ; H^1 (\mathbbm{R}^2)) \cap C (0, T ; L^2 (\mathbbm{R}^2))$
        which solves (\ref{macroscopic-eq-regularised}) in the sense of
        distributions.
        
        \item The Keller-Segel equation (\ref{macroscopic-eq}) has a unique weak
        non-negative solution $\rho \in L^{\infty} \left( \mathbbm{R}_+ ;
        {\nobreak} L^1 (\mathbbm{R}^2) \right)$ satisfying the conservation of
        mass
        \[ \int_{\mathbbm{R}^2} \rho \mathd x = \int_{\mathbbm{R}^2} \rho_0 \mathd
           x \quad (\equallim 1), \]
        the second moment equation
        \[ \int_{\mathbbm{R}^2} \rho (t, x)  | x |^2 \mathd x = 4 \left( 1 -
           \frac{\chi}{8 \pi} \right) t + \int_{\mathbbm{R}^2} \rho_0 (x)  | x |^2
           \mathd x \]
        and the free energy inequality
        \[ \mathcal{F} [\rho (t)] + \int_0^t \int_{\mathbbm{R}^2} \rho \left|
           \nabla (\log \rho) + \chi  (k \ast \rho) \right|^2 \mathd x \mathd s
           \leqslant \mathcal{F} [\rho_0], \]
        where the free energy $\mathcal{F}$ is given by
        \[ \mathcal{F} [\rho] \assign \int_{\mathbbm{R}^2} \rho \log \rho \mathd x
           - \frac{\chi}{2} \int_{\mathbbm{R}^2} \rho (\phi \ast \rho) \mathd x.
        \]
        \item The sequence $(\rho^N)$ of solutions of (\ref{macroscopic-eq-regularised}) converges weakly to the solution $\rho$ of the Keller-Segel
        equation (\ref{macroscopic-eq}).
      \end{enumerateroman}
    \end{proposition}
    
    We refer to {\cite{blanchet_two-dimensional_2006}} and
    {\cite{fernandez_uniqueness_2016}} for the proof. More precisely, the
    existence of the sequence $\rho^N$ and the weak convergence of a subsequence
    of $\rho^N$ to a weak solution of the Keller-Segel equation (\ref{macroscopic-eq}) were proved in {\cite{blanchet_two-dimensional_2006}}. Together with the
    uniqueness of the weak solution $\rho$ of (\ref{macroscopic-eq}), which was
    proved in {\cite{fernandez_uniqueness_2016}}, it follows the weak convergence
    of the whole sequence $\rho^N$ (and not just a subsequence) to this unique
    solution $\rho$.
    
    For the proof of Proposition \ref{macro-existence-and-convergence} only
    $\rho_0 \in L^1 (\mathbbm{R}^2, (1 + | x |^2) \mathd x)$, and not $\rho_0 \in
    L^1 \left( \mathbbm{R}^2, \left( 1 + {\nobreak} | x |^2 \right) \mathd x
    \right) \cap L^{\infty} (\mathbbm{R}^2) \cap H^2 (\mathbbm{R}^2)$ as required
    in condition (\ref{condition-initial-data}), is necessary. If moreover the
    initial density is bounded in $L^{\infty}$ we find in Proposition \ref{Linf-estimates} that the solutions of the Keller-Segel and the regularised
    Keller-Segel equations are uniformly bounded in $L^{\infty}$ as well. Finally
    with the full condition $\rho_0 \in L^1 (\mathbbm{R}^2, (1 + | x |^2) \mathd
    x) \cap L^{\infty} (\mathbbm{R}^2) \cap H^2 (\mathbbm{R}^2)$ we prove some
    H{\"o}lder estimates in Proposition \ref{hˆlder-estimates}. The proofs of
    these two last propositions are contained in Section \ref{proofs-of-estimates}.
    
    \begin{proposition}
      \label{Linf-estimates}{\tmem{($L^{\infty}$ estimates)}} Assume $\chi$ and
      $\rho_0$ satisfy condition (\ref{condition-initial-data}). Then for each $T
      > 0$ there exists a positive constant $C$ such that
      \[ \sup_{t \in [0, T]} \| \rho^N_t \|_{\infty}, \sup_{t \in [0, T]} \|
         \rho_t \|_{\infty} \leqslant C \]
      holds for the solution $\rho^N$ of (\ref{macroscopic-eq-regularised}) and
      the solution $\rho$ of (\ref{macroscopic-eq}).
    \end{proposition}
    
    \begin{proposition}
      \label{hˆlder-estimates}{\tmem{(H{\"o}lder estimates)}} Assume $\chi$ and
      $\rho_0$ satisfy condition (\ref{condition-initial-data}). Then for each $T
      > 0$ there exist positive constants $C_1$ and $C_2$ depending on $\rho_0$
      and $T$, such that for each $N \in \mathbbm{N}$ and $t \in [0, T]$ the
      following estimates hold for the solution $\rho^N$ of (\ref{macroscopic-eq-regularised}) and the solution $\rho$ of (\ref{macroscopic-eq}):
      \begin{enumerateroman}
        \item \label{hˆlder}$[\rho^N (t)]_{0, \alpha}, [\rho (t)]_{0, \alpha}
        \leqslant C_1$,
        
        \item \label{lipschitz}$[k^N \ast \rho^N (t)]_{0, 1}, [k \ast \rho
        (t)]_{0, 1} \leqslant C_2$.
      \end{enumerateroman}
    \end{proposition}
    
    \subsection{Microscopic equations}
    
    We focus first on the interacting $N$-particle system (\ref{particle-system})
    and its regularised version (\ref{microscopic-eq-regularised}). Since for each
    $N > 0$ the kernel of (\ref{microscopic-eq-regularised}) is globally Lipschitz
    continous, the solution of (\ref{microscopic-eq-regularised}) is strongly and
    uniquely well-defined. For the original singular situation (\ref{particle-system}) it is much more delicate. Cattiaux and P{\'e}d{\`e}ches
    {\cite[Theorem 1.5]{cattiaux_2-d_2016}} proved the following existence and
    uniqueness result:
    
    \begin{proposition}
      {\tmem{(Existence and uniqueness) }}Let $\mathcal{M} \assign \left\{ \right.$ there exists at most one pair $i \neq j$ such that $\left. X^i = X^j \right\}$. Then, for $N \geqslant 4$ and $\chi < 8 \pi \left( 1 - \frac{1}{N
      - 1} \right)$ there exists a unique (in distribution) non explosive solution
      of (\ref{particle-system}) starting from any $x \in \mathcal{M}$.
    \end{proposition}
    
    We continue with the mean-field $N$-particle system (\ref{microscopic-mean}),
    its regularised version (\ref{microscopic-mean-regularised}) and its
    regularised and linearised version (\ref{microscopic-mean-mixed}). According
    to Proposition \ref{hˆlder-estimates} the mean-field force $\overline{K}^N$ is
    Lipschitz in the space variable, uniformly in $t \in [0, T]$ and $N \in
    \mathbbm{N}$. Therefore, the linear equation (\ref{microscopic-mean-mixed})
    has a unique strong solution. For the existence and uniqueness of solutions of
    the non-linear equations (\ref{microscopic-mean}) and (\ref{microscopic-mean-regularised}) we refer to {\cite[Theorem 2.2]{liu_random_2015}}.
    
    \section{Preliminary results}\label{preliminary-results}
    
    \subsection{Local Lipschitz bound for the regularised interaction force}
    
    The regularised interaction force $K^N$ defined in (\ref{regularised-force})
    is locally Lipschitz, with a local Lipschitz bound depending on $N$. The proof
    of this statement is conducted in the following Lemma, which is formulated to
    include more general cutoffs that we will need to consider in this paper.
    
    \begin{lemma}
      \label{lipschitz-bound-G}Let $\nu = \nu (N)$ be a monotone increasing
      function of $N$ with $\lim_{N \rightarrow \infty} \nu (N) = \infty$, and
      consider the force $k^{\nu}$ with cutoff at $\nu^{} (N)^{- 1}$, $k^{\nu} (x)
      \assign - \nabla (\phi^1 (\nu^{} x))$ for the bump function $\phi^1$ defined
      in Section \ref{equations}, meaning in particular that $k^{\nu} (x)
      \leqslant (2 \pi | x |)^{- 1}$ and
      \[ k^{\nu} (x) = \left\{ \begin{array}{ll}
           \frac{x}{2 \pi | x |^2}, & | x | \geqslant 2 \nu^{- 1},\\
           0, & | x | \leqslant \nu^{- 1}
         \end{array} \right. . \]
      \begin{enumerateroman}
        \item \label{lipschitz-bound-2D}For each $x, y \in \mathbbm{R}^2$ with $|
        x - y | \leqslant 2 \nu^{- 1}$ it holds
        \[ | k^{\nu} (x) - k^{\nu} (y) | \leqslant l^{\nu} (y)  | x - y |, \]
        where
        \[ l^{\nu} (y) \assign \left\{ \begin{array}{ll}
             \frac{16}{| y |^2}, & | y | \geqslant 4 \nu^{- 1},\\
             \nu^2, & | y | \leqslant 4 \nu^{- 1}
           \end{array} \right. . \]
        \item \label{lipschitz-bound-2ND}Let the resulting force be $K^{\nu}_i
        (x_1, \ldots, x_N) \assign - \frac{\chi}{N}  \sum_{j \neq i} k^{\nu} (x_i
        - x_j)$ and define
        \[ L^{\nu}_i (y_1, \ldots, y_N) \assign - \frac{\chi}{N}  \sum_{j \neq i}
           l^{\nu} (y_i - y_j) . \]
        Then, for each $x, y \in \mathbbm{R}^{2 N}$ with $| x - y |_{\infty}
        \leqslant \nu^{- 1}$ it holds
        \[ | K_i^{\nu} (x) - K_i^{\nu} (y) | \leqslant 2 L_i^{\nu} (y)  | x - y
           |_{\infty} . \]
      \end{enumerateroman}
    \end{lemma}
    
    \begin{proof}
      (\ref{lipschitz-bound-2D}) By the Mean Value Theorem the bound
      \[ | k^{\nu} (x) - k^{\nu} (y) | \leqslant | D k^{\nu} (z) |  | x - y | \]
      holds for some point $z$ in the segment which joins $x$ and $y$. We
      distinguish between the following two cases:
      
      \
      
      {\noindent}{\tmem{Case 1}}: $| y | \leqslant 4 \nu^{- 1}$.
      
      Since the derivative of $k^{\nu}$ is globally bounded by $\nu^2 / \pi$, and
      consequently by $\nu^2$ as well, it follows that
      \[ | k^{\nu} (x) - k^{\nu} (y) | \leqslant \| D k^{\nu} \|  | x - y |
         \leqslant l^{\nu} (y) |x - y|. \]

      {\noindent}{\tmem{Case 2}}: $| y | \geqslant 4 \nu^{- 1}$.
      
      Since $| z - y | \leqslant | x - y | \leqslant 2 \nu^{- 1}$, it follows that
      $| z | \geqslant 2 \nu^{- 1}$. This means in particular that the derivative
      of $k^{\nu}$ at $z$ is bounded by $| z |^{- 2} / \pi$ and also that $| z - y
      | \leqslant | z |$, so
      \[ | y |^2 \leqslant (| y - z | + | z |)^2 \leqslant (2 | z |)^2 = 4 | z |^2
         . \]
      Therefore,
      \begin{eqnarray*}
        | k^{\nu} (x) - k^{\nu} (y) | & \leqslant & | D k^{\nu} (z) |  | x - y |\\
        & \leqslant & 2^{- 1}  | z |^{- 2}  | x - y |\\
        & \leqslant & 2 | y |^{- 2}  | x - y |\\
        & \leqslant & l^{\nu} (y) |x - y|.
      \end{eqnarray*}
      Finally, (\ref{lipschitz-bound-2ND}) follows directly from (\ref{lipschitz-bound-2D}).
    \end{proof}
    
    \subsection{Law of large numbers}
    
    In the proof of the main theorem we define several ``exceptional'' sets and
    rely on the fact that the measure of these sets is exponentially small. This
    fact is proven in the next Proposition, a {\tmem{law of large numbers}} for
    our setting, for all these sets are events where the sample mean and expected
    values of some family of independent variables are not close. The steps we
    follow for this version of the law of large numbers are the standard ones, the
    only issue being that the $k$-th moments of the variables we consider are not
    bounded but instead grow with $N$ to infinity. We'll see that their growth is
    nevertheless slow enough and we still obtain a rate of convergence which is
    faster than $C_{\gamma} N^{- \gamma}$ for any $\gamma > 0$, where $C_{\gamma}
    > 0$ is a constant depending on the choice of $\gamma$ but not on $N$.
    
    \begin{proposition}
      \label{loln-general}{\tmem{(Law of large numbers)}} Let $\alpha, \delta > 0$
      be such that $\alpha + \delta < 1 / 2$. For $N \in \mathbbm{N}$ let $Z^1,
      \ldots, Z^N$ be $N$ independent random variables in $\mathbbm{R}^2$ and
      assume that $Z^i$ has a probability density that we denote by $u^i$, $i = 1,
      \ldots, N$. Let $h = (h^1, h^2) : \mathbbm{R}^2 \rightarrow \mathbbm{R}^2$
      be a continuous function satisfying $| h (x) | \leqslant C_h \min \{
      N^{\alpha}, | x |^{- 1} \}$. Define $H_i (Z) = (H^1_i (Z), H^2_i (Z))
      \assign \frac{1}{N} \sum_{j \neq i} h (Z^i - Z^j)$ and the following sets
      \[ S \assign \{ \sup_{1 \leqslant i \leqslant N} | H_i (Z) -\mathbbm{E} (H_i
         (Z)) | \geqslant N^{- (\alpha + \delta)} \}, \]
      \[ \tilde{S} \assign \{ \sup_{1 \leqslant i \leqslant N} | H_i (Z)
         -\mathbbm{E}_{(- i)} (H_i (Z)) | \geqslant N^{- (\alpha + \delta)} \}, \]
      where $\mathbbm{E}_{(- i)}$ stands for the expectation with respect to every
      variable but $Z^i$, that is, $\mathbbm{E}_{(- i)} (H_i (Z)) = \frac{1}{N}
      \sum_{j \neq i} (h \ast u^j) (Z^i)$.
      
      Define $\varepsilon \assign 1 - 2 (\alpha + \delta)$ (strictly positive by
      assumption) and assume that, for each $i$,
      \begin{equation}
        \log N \| u^i \|_{\infty} + \| u^i \|^2_{\infty} \leqslant C_0
        N^{\varepsilon / 2} \label{loln-density-condition}
      \end{equation}
      holds for some constant $C_0$ independent of $N$ and $i$. Then, for each
      $\gamma > 0$ there exists a constant $C_{\gamma}$ (depending on $\gamma$,
      $\varepsilon$, $C_0$ and $C_h$) such that
      \[ \mathbbm{P} (S), \mathbbm{P} (\tilde{S}) \leqslant C_{\gamma} N^{-
         \gamma} . \]
    \end{proposition}
    
    \begin{proof}
      Because we can replace $\mathbbm{E} (H_i (Z))$ by $\mathbbm{E}_{(- i)} (H_i
      (Z))$ in the proof, it is enough to prove the statement for the first set
      $S$. Also notice that since
      \[ \mathbbm{P} (\sup_{1 \leqslant i \leqslant N} | H_i (Z) -\mathbbm{E} (H_i
         (Z)) | \geqslant N^{- (\alpha + \delta)}) \leqslant \sum_{i = 1, \nu
         \equallim 1}^{N, 2} \mathbbm{P} (| H^{\nu}_i (Z) -\mathbbm{E} (H^{\nu}_i
         (Z)) | \geqslant N^{- (\alpha + \delta)}) \]
      holds, it suffices to prove that
      \[ \mathbbm{P} (| H^{\nu}_i (Z) -\mathbbm{E} (H^{\nu}_i (Z)) | \geqslant
         N^{- (\alpha + \delta)}) \leqslant C_{\gamma} N^{- \gamma} \]
      for each $\gamma > 0$, each $i = 1, \ldots N$ and $\nu = 1, 2$. Let then
      $\gamma > 0$, $\nu \in \{ 1, 2 \}$ and let us for simplicity take $i = 1$. \
      
      We use Markov's inequality of order $2 m$ and determine later the right
      choice of $m$ for the given $\gamma$ and the quantity $(\alpha + \delta)$ in
      the exponent of the allowed error $N^{- (\alpha + \delta)}$. For $j = 2,
      \ldots, N$ let us denote by $\Theta_j$ the (independent) random variables
      $\Theta_j \assign h^{\nu} (Z^1 - Z^j)$ and by $\mu_j$ its expected value
      \[ \mu_j \assign \int h^{\nu} (z_1 - z_j) u^1 (z_1) u^j (z_j) \mathd z_1
         \mathd z_j . \]
      Now by Markov's inequality
      \begin{eqnarray*}
        \mathbbm{P} (| H_1 (Z) -\mathbbm{E} (H_1 (Z)) | \geqslant N^{- (\alpha +
        \delta)}) & = & \mathbbm{P} \left( \frac{1}{N} \left| \sum_{j \neq 1}
        (\Theta_j - \mu_j) \right| \geqslant N^{- (\alpha + \delta)} \right)\\
        & \leqslant & N^{2 (\alpha + \delta) m} \mathbbm{E} \left( \left(
        \frac{1}{N}  \sum_{j \neq 1} (\Theta_j - \mu_j) \right)^{2 m} \right) .
      \end{eqnarray*}
      The expectation on the right hand side can be estimated by using the
      multinomial formula
      \[ (x_2 + \ldots + x_N)^{2 m} = \sum_{a_2 + \ldots + a_n = 2 m} C_a 
         \prod_{j = 2}^N x_j^{a_j}, \]
      where $a = (a_2, \ldots, a_N)$ is a multiindex and $C_a = \binom{2 m}{a_2,
      \ldots, a_N} = \frac{(2 m) !}{a_2 ! \ldots a_N !}$. Consequently
      \[ \mathbbm{E} \left( \left( \frac{1}{N}  \sum_{j \neq 1} (\Theta_j - \mu_j)
         \right)^{2 m} \right) = N^{- 2 m} \sum_{a_2 + \ldots + a_N = 2 m} C_a
         \prod_{j \neq 1} \mathbbm{E} ((\Theta_j - \mu_j)^{a_j}) . \]
      Here note that if $a_j = 1$ for some $j$ then the whole term is zero, since
      $\mathbbm{E} ((\Theta_j - \mu_j)) = 0$. Therefore we are left only with
      terms with at most $m$ non-zero entries. If we denote by $| a |$ the number
      of non-zero entries of the multiindex $a$, the sum above simplifies to
      \[ \mathbbm{E} \left( \left( \frac{1}{N}  \sum_{j \neq 1} (\Theta_j - \mu_j)
         \right)^{2 m} \right) = N^{- 2 m} \sum_{\underset{| a | \leqslant
         m}{\scriptsize{a_2 + \ldots + a_N = 2 m}}} C_a \prod_{j \neq 1}
         \mathbbm{E} ((\Theta_j - \mu_j)^{a_j}) . \]
      Next we estimate the $a_j$-th order moment of $\Theta_j$, for $a_j \leqslant
      2 m$: specifically we prove that
      \[ \mathbbm{E} ((\Theta_j - \mu_j)^{a_j}) \leqslant C^{a_j}_h C_0 N^{\alpha
         (a_j - 2) + \varepsilon / 2} . \]
      The $a_j$-th order moment of $\Theta_j$ equals
      \[ \int_{\mathbbm{R}^2} (h^{\nu} (z_1 - z_j) - \mu_j)^{a_j} u^1 (z_1) u^j
         (z_j) \mathd z_1 \mathd z_j . \]
      We factor the power in the integrand as
      \[ (h^{\nu} (z_1 - z_j) - \mu_j)^{a_j} = (h^{\nu} (z_1 - z_j) - \mu_j)^{a_j
         - 2}  (h^{\nu} (z_1 - z_j) - \mu_j)^2, \]
      then estimate the term to the power $a_j - 2$ by its supremum norm and
      integrate only the second factor. It holds that
      \begin{eqnarray*}
        \| h^{\nu} (z_1 - z_j) - \mu_j \|_{\infty} & \leqslant & C \| h^{\nu}
        \|_{\infty} + \| (h^{\nu} \ast u^j) \|_{\infty}\\
        & \leqslant & C \| h \|_{\infty} \leqslant C_h N^{\alpha} .
      \end{eqnarray*}
      After integrating the term to the second power we find
      \begin{eqnarray*}
  \int_{\mathbbm{R}^2} (h^{\nu} (z_1 - z_j) - \mu_j)^2 u^1 (z_1) u^j (z_j)
  \mathd z_1 \mathd z_j & = & \mu^2_j + 2 \mu_j  \int h^{\nu} (z_1 - z_j) u^1
  (z_1) u^j (z_j) \mathd z_1 \mathd z_j\\
  &  & + \int_{\mathbbm{R}^2} h^{\nu} (z_1 - z_j)^2 u^1 (z_1) u^j (z_j)
  \mathd z_1 \mathd z_j\\
  & \leqslant & 3 \| h \ast u^j \|_{\infty}^2 + \| h^2 \ast u^j \|_{\infty}\\
  & \leqslant & C_h  (\| u^j \|_{\infty}^2 + \log N \| u^j \|_{\infty})\\
  & \leqslant & C_h C_0 N^{\varepsilon / 2} .
\end{eqnarray*}
      Altogether
      \begin{eqnarray*}
        \mathbbm{E} ((\Theta_j - \mu_j)^{a_j}) & = & \int_{\mathbbm{R}^2} (h^{\nu}
        (z_1 - z_j) - \mu_j)^{a_j} u^1 (z_1) u^j (z_j) \mathd z_1 \mathd z_j\\
        & \leqslant & \| h^{\nu} (z_1 - z_j) - \mu_j \|^{a_j - 2}_{\infty} 
        \int_{\mathbbm{R}^2} (h^{\nu} (z_1 - z_j) - \mu_j)^2 u^1 (z_1) u^j (z_j)
        \mathd z_1 \mathd z_j\\
        & \leqslant & C_h^{a_j} C_0 N^{\alpha (a_j - 2) + \varepsilon / 2} .
      \end{eqnarray*}
      Let now $k \leqslant m$ and consider only the multiindices $a$ with $k$
      non-zero entries, that is with $| a | = k$. It holds
      \begin{eqnarray*}
        \sum_{\underset{| a | = k}{\scriptsize{a_2 + \ldots + a_N = 2 m}}} C_a
        \prod_{j \neq 1} \mathbbm{E} ((\Theta_j - \mu_j)^{a_j}) & \leqslant &
        \sum_{\underset{| a | = k}{\scriptsize{a_2 + \ldots + a_N = 2 m}}} C_a
        C_h^{2 m} C_0^k N^{\alpha (2 m - 2 k) + \varepsilon k / 2}\\
        & \leqslant & \sum_{\underset{| a | = k}{\scriptsize{a_2 + \ldots + a_N =
        2 m}}} (2 m)^{2 m} C_h^{2 m} C_0^m N^{\alpha (2 m - 2 k) + \varepsilon k /
        2},
      \end{eqnarray*}
      \tmfoldedenv{where we used that $C_a = \binom{2 m}{a_2, a_3, \ldots, a_N}
      \leqslant (2 m)^{2 m}$. Since the number of terms in the sum, i.e. the
      number of ways of choosing $k$ numbers that add up $2 m$ counting all
      permutations, is bounded by $N^k  (2 m)^k$, we find that}{Assume $k$ indices
      $j_1, \ldots, j_k$ from $\{ 2, \ldots, N \}$ are given. It is now left to
      choose values for $a_{j_1}, \ldots, a_{j_k}$ such that  $a_{j_1} + \ldots +
      a_{j_k} = 2 m$.
      
      In how many ways can we add up $2 m$ with $k$ terms (each permutation of the
      terms counts)? An upper bound for this is $(2 m)^k$
      
      Now, in how many ways can we choose $k$ elements from $\{ 2, \ldots, N \}$?
      This is $\binom{N}{k} \leqslant N^k$.
      
      Therefore, we can estimate the number of terms in $\sum_{\underset{| a | =
      k}{\scriptsize{a_2 + \ldots + a_N = 2 m}}}$by $N^k  (2 m)^k$.}
      \begin{eqnarray}
        \sum_{\underset{| a | = k}{\scriptsize{a_2 + \ldots + a_N = 2 m}}} C_a
        \prod_{j \neq 1} \mathbbm{E} ((\Theta_j - \mu_j)^{a_j}) & \leqslant & (2
        m)^{3 m} C_h^{2 m} C_0^m N^{\alpha (2 m - 2 k) + \varepsilon k / 2} N^k
        \nonumber\\
        & \leqslant & C_m N^{2 m \alpha} N^{k (1 - 2 \alpha + \varepsilon / 2)}, 
        \label{loln-sum-estimate}
      \end{eqnarray}
      for a constant $C_m > 0$ only depending on $m$, $C_h$ and $C_0$. At this
      point we can estimate the desired expected value
      \begin{eqnarray*}
        \mathbbm{E} \left( \left( \frac{1}{N}  \sum_{j \neq 1} (\Theta_j - \mu_j)
        \right)^{2 m} \right) & = & N^{- 2 m}  \sum_{\underset{| a | \leqslant
        m}{\scriptsize{a_2 + \ldots + a_N = 2 m}}} C_a \prod_{j \neq 1}
        \mathbbm{E} ((\Theta_j - \mu_j)^{a_j})\\
        & = & N^{- 2 m}  \sum_{k = 1}^m \sum_{\underset{| a | =
        k}{\scriptsize{a_2 + \ldots + a_N = 2 m}}} C_a \prod_{j \neq 1}
        \mathbbm{E} ((\Theta_j - \mu_j)^{a_j})\\
        & \leqslant & C_m N^{- 2 m}  \sum_{k = 1}^m N^{2 m \alpha} N^{k (1 - 2
        \alpha + \varepsilon / 2)}\\
        & \leqslant & C_m N^{- 2 m} N^{m (2 \alpha + 1 - 2 \alpha + \varepsilon /
        2)}\\
        & \leqslant & C_m N^{- m (1 - \varepsilon / 2)},
      \end{eqnarray*}
      where we used (\ref{loln-sum-estimate}) and the positivity of $1 - 2 \alpha
      + \varepsilon / 2$. Finally we find that
      \begin{eqnarray*}
        \mathbbm{P} (| H_1 (Z) -\mathbbm{E} (H_1 (Z)) | \geqslant N^{- (\alpha +
        \delta)}) & \leqslant & N^{2 (\alpha + \delta) m} \mathbbm{E} \left(
        \left( \frac{1}{N}  \sum_{j \neq 1} (\Theta_j - \mu_j) \right)^{2 m}
        \right)\\
        & \leqslant & C_m N^{2 (\alpha + \delta) m} N^{- m (1 - \varepsilon /
        2)}\\
        & = & C_m N^{- m (1 - 2 (\alpha + \delta) - \varepsilon / 2)}\\
        & \leqslant & C_m N^{- m \varepsilon / 2} = \tilde{C}_{\gamma} N^{-
        \gamma}
      \end{eqnarray*}
      holds for $m = 2 \gamma / \varepsilon$, where $\tilde{C}_{\gamma} \assign
      C_{2 \gamma / \varepsilon}$ depends on $\gamma$, $\varepsilon$, $C_0$ and
      $C_h$.
    \end{proof}
    
    \subsection{Comparison of solutions of (\ref{microscopic-mean-mixed}) starting
    at different points}
    
    In this section we address the following question: how different is the action
    of the force $K^N$ on two solutions of (\ref{microscopic-mean-mixed}) that
    start at different points? An estimate of this difference will be very useful
    in the second case (for large times) of the proof of the main theorem and
    innovates the methods presented in {\cite{boers_mean_2016}} and
    {\cite{lazarovici_mean-field_2015}}. The estimate is provided in Corollary
    \ref{corollary-to-lemma}. Recall that for each $x \in \mathbbm{R}^{2 N}$,
    $Z_{t, s}^{x, N} \in \mathbbm{R}^{2 N}$ denotes the process starting at point
    $x$ at time $s$ and evolving for times greater that $s$ according to the
    mean-field force $\overline{K}^N$. That is, $Z_{t, s}^{x, N}$ solves
    (\ref{microscopic-mean-mixed}) with constant initial condition $x$ and initial
    time $s$. Furthermore $Z_{t, s}^{x, N}$ has the {\tmem{strong Feller
    property}}, implying in particular that it has a transition probability
    density. Since the processes $Z_{t, s}^{x, 1}, \ldots, Z_{t, s}^{x, N}$ are
    independent, the joint transition probability density $u_{t, s}^{x, N} (z_1,
    \ldots, z_N)$ is given by the product $u_{t, s}^{x, N} (z_1, \ldots, z_N)
    \assign \prod u_{t, s}^{x, i, N} (z_i)$. Here each term $u_{t, s}^{x, i, N}$
    is the transition probability density of $Z_{t, s}^{x, i, N}$ and also the
    solution of the {\tmem{linearised Keller-Segel equation}}
    \begin{equation}
      \partial_t u_{t, s}^{x, i, N} = \Delta u_{t, s}^{x, i, N} - \nabla \cdot
      (f_t^N u_{t, s}^{x, i, N}), \qquad u_{s, s}^{x, i, N} = \delta_{x_i},
      \label{macroscopic-eq-linearised}
    \end{equation}
    where $f_t^N \assign \chi k^N \ast \rho_t^N$ and $\rho_t^N$ solves the
    regularised Keller-Segel equation (\ref{macroscopic-eq-regularised}) with
    initial condition $\rho_0$. Consider now the processes $Z_{t, s}^{x, N}$ and
    $Z_{t, s}^{y, N}$ for two different starting points $x, y \in \mathbbm{R}^{2
    N}$. It is intuitively clear that the probability densities $u_{t, s}^{x, N}$
    and $u_{t, s}^{y, N}$ are just a shift of each other. The next lemma gives an
    estimate for the $L^{\infty}$ norm of each $u_{t, s}^{x, N}$ as well as for
    the distance in $L^{\infty}$ between any two densities $u_{t, s}^{x, N}$ and
    $u_{t, s}^{y, N}$ in terms of the distance between the starting points $x$ and
    $y$ and the elapsed time $t - s$.
    
    \
    
    \begin{lemma}
      \label{lemma}There exists a positive constant $C$ depending on $\rho_0$ and
      $T$ such that for each $N \in \mathbbm{N}$, any starting points $x, y \in
      \mathbbm{R}^{2 N}$ and any time $0 < t \leqslant T$ the following estimates
      for the transition probability densities $u_{t, s}^{x, N}$ resp. $u_{t,
      s}^{y, N}$ of the processes $Z_{t, s}^{x, N}$ resp. $Z_{t, s}^{y, N}$ given
      by (\ref{microscopic-mean-mixed}) hold:
      \begin{enumerateroman}
        \item $\label{density-starting-at-delta} \| u_{t, s}^{x, N} \|_{\infty}
        \leqslant C ((t - s)^{- 1} + 1)$,
        
        \item \label{distance-densities-starting-at-deltas}$\| u_{t, s}^{x, N} -
        u_{t, s}^{y, N} \|_{\infty} \leqslant C ((t - s)^{- 3 / 2} + 1)  | x - y
        |_{\infty}$.
      \end{enumerateroman}
    \end{lemma}
    
    \begin{proof}
      Both estimates can be proved in the same way. We just give the proof for
      part (\ref{distance-densities-starting-at-deltas}), which can be easily
      adapted for part (\ref{density-starting-at-delta}). For simplicity of
      notation we assume $s = 0$ and write simply $u^{x_i}_t$ instead of $u^{x, i,
      N}_{t, 0}$. What we need to show is then that
      \[ \| u_t^{x_i} - u_t^{y_i} \|_{\infty} \leqslant C (t^{- 3 / 2} + 1)  | x_i
         - y_i | \]
      holds for each $i = 1, \ldots, N$ and for a constant $C > 0$ depending only
      on $\rho_0$ and $T$. We show this inductively.
      
      Let us then fix $i \in \{ 1, \ldots, N \}$ and define $v_t \assign
      u_t^{x_i} - u_t^{y_i}$. For a solution of (\ref{macroscopic-eq-linearised})
      we see that
      \begin{eqnarray}
        u_t^{x_i} & = & G (t) \ast \delta_{x_i} - \int_0^t G (t - s) \ast
        \tmop{div} (u_s^{x_i} f^N_s) \mathd s \nonumber\\
        & = & G (t) \ast \delta_{x_i} - \int_0^t \nabla G (t - s) \ast (u_s^{x_i}
        f^N_s) \mathd s,  \label{variation-of-constants}
      \end{eqnarray}
      where $G (t, x) \assign \frac{1}{2 \pi t} \mathe \tmop{xp} \left( - \frac{|
      x |^2}{2 t} \right)$ denotes the heat kernel. By substracting the
      corresponding equations for $u_t^{x_i}$ and $u_t^{y_i}$ it follows
      \[ v_t = G (t) \ast (\delta_{x_i} - \delta_{y_i}) - \int_0^t \nabla G (t -
         s) \ast (v_s f^N_s) \mathd s \]
      and consequently, for $p \in [1, \infty]$,
      \[ \| v_t \|_p \leqslant \| G (t) \ast (\delta_{x_i} - \delta_{y_i}) \|_p +
         \int_0^t \| \nabla G (t - s) \ast (v_s f^N_s) \|_p \mathd s \]
      holds due to Bochner's Theorem . Next we use Young's inequality for
      convolutions\footnote{For two functions $a, b : \mathbbm{R}^n \rightarrow
      \mathbbm{R}$ and exponents $p, q, r \in [1, \infty]$ satisfying $1 +
      \frac{1}{p} = \frac{1}{r} + \frac{1}{q}$ it holds
      \[ \| a \ast b \|_p \leqslant \| a \|_r  \| b \|_q . \]}. We split the last
      integral into two parts and use Young's inequality with different exponents
      for each part
      \begin{eqnarray}
        \int_0^t \| \nabla G (t - s) \ast (v_s f^N_s) \|_p \mathd s & = &
        \int_0^{t / 2} \| \nabla G (t - s) \ast (v_s f^N_s) \|_p \mathd s
        \nonumber\\
        &  & + \int_{t / 2}^t \| \nabla G (t - s) \ast (v_s f^N_s) \|_p \mathd s
        \nonumber\\
        & \leqslant & C \int_0^{t / 2} \| \nabla G (t - s) \|_p  \| v_s \|_1
        \mathd s \nonumber\\
        &  & + C \int_{t / 2}^t \| \nabla G (t - s) \|_{3 / 2}  \| v_s \|_q
        \mathd s,  \label{integral-splitting}
      \end{eqnarray}
      where $C \assign \sup_{0 \leqslant t \leqslant T} \| f_t^N \|_{\infty}$ is
      finite since $\| \rho_t^N \|_1$ is equal to $\| \rho_0 \|_1$ and by
      Proposition \ref{Linf-estimates} $\| \rho_t^N \|_{\infty}$ is also uniformly
      bounded in $t \in [0, T]$ and $N \in \mathbbm{N}$. The choice of the
      exponent $r = 3 / 2$ for the norm of $\nabla G$ in the second integral is as
      good as any other choice $r \in (1, 2)$ since we just need the term $\|
      \nabla G \|_r$ to be integrable in $\left[ 0, t \right]$. Observe that with
      the previous bound for $\| v_t \|_p$ and taking $p_n \assign q$ and $p_{n +
      1} \assign p$ in (\ref{integral-splitting}) we find
      \begin{eqnarray}
        \| v_t \|_{p_{n + 1}} & \leqslant & \| G (t) \ast (\delta_{x_i} -
        \delta_{y_i}) \|_{p_{n + 1}} + C \int_0^{t / 2} \| \nabla G (t - s)
        \|_{p_{n + 1}}  \| v_s \|_1 \mathd s \nonumber\\
        &  & + C \int_{t / 2}^t \| \nabla G (t - s) \|_{3 / 2}  \| v_s \|_{p_n}
        \mathd s,  \label{iteration}
      \end{eqnarray}
      where the relation between the exponents $p_{n + 1} = 3 \frac{p_n}{3 - p_n}$
      follows from Young's inequality. Therefore, if we are able to estimate $\|
      v_t \|_1$ we can then iteratively estimate the $L^p$ norms of $v_t$ for
      higher exponents. Since the function $x \mapsto 3 \frac{x}{3 - x}$ on $[0,
      3)$ is strictly monotone increasing, grows to infinity as $x$ approaches $3$
      and its first derivative is non-decreasing, it is already clear that
      starting at $p_1 = 1$ the exponent $p_k = \infty$ must be attained after a
      finite number $k$ of steps. Specifically, if we take $p_1 = 1$, we reach the
      desired norm $\| v_t \|_{\infty}$ after $k = 4$ steps. Below we go through
      the first two steps in detail, the last two can be completed analogously. We
      will need some estimates for the $L^p$ norms of the heat kernel $G$ and its
      derivative, which are given in Lemma \ref{p-norm-estimates-heat}.
      
      \
      
      {\tmem{Step $k = 1$, $p_1 = 1$:}} We compute the first norm directly using
      a Gr{\"o}nwall-type inequality.
      \begin{eqnarray*}
        \| v_t \|_1 & \leqslant & \| G (t, \cdot - x_0) - G (t, \cdot - y_0) \|_1
        + \int_0^t \| \nabla G (t - s) \ast (v_s f^N_s) \|_p \mathd s\\
        & \leqslant & \| G (t, \cdot - x_0) - G (t, \cdot - y_0) \|_1 + \int_0^t
        \| \nabla G (t - s) \|_1  \| v_s \|_1  \| f^N_s \|_{\infty} \mathd s\\
        & \leqslant & C \frac{| x_0 - y_0 |}{t^{1 / 2}} + C \int_0^t (t - s)^{- 1
        / 2}  \| v_s \|_1 \mathd s.
      \end{eqnarray*}
      By Gr{\"o}nwall's inequality we find
      \begin{eqnarray*}
        \| v_t \|_1 & \leqslant & C \frac{| x_0 - y_0 |}{t^{1 / 2}} + C | x_0 -
        y_0 |  \int_0^t s^{- 1 / 2}  (t - s)^{- 1 / 2} \mathe^{C \int_s^t (t -
        \sigma)^{- 1 / 2} \mathd \sigma} \mathd s\\
        & \leqslant & C \frac{| x_0 - y_0 |}{t^{1 / 2}} + C \mathe^{Ct^{1 / 2}} 
        | x_0 - y_0 | .
      \end{eqnarray*}
      Here we used that the integral $\int_0^t s^{- 1 / 2}  (t - s)^{- 1 / 2}$ is
      finite since it can be split into
      \[ \int_0^t s^{- 1 / 2}  (t - s)^{- 1 / 2} \mathd s = \int_0^{t / 2} s^{- 1
         / 2}  (t - s)^{- 1 / 2} \mathd s + \int_{t / 2}^t s^{- 1 / 2}  (t - s)^{-
         1 / 2} \mathd s, \]
      and both terms are finite. Consequently
      \[ \| v_t \|_1 \leqslant C (t^{- 1 / 2} + 1)  | x_0 - y_0 | \]
      holds for a constant $C$ depending only on $\sup_{0 \leqslant t \leqslant T}
      \| f_t^N \|_{\infty}$.
      
      \
      
      {\tmem{Step $k = 2$, $p_2 = \frac{3}{2}$:}} Recall that the next exponent
      is computed via the relationship $p_{n + 1} = 3 \frac{p_n}{3 - p_n}$. In
      this and the following steps we just need to substitute the found estimates
      into (\ref{iteration}):
      \begin{eqnarray*}
        \| v_t \|_{3 / 2} & \leqslant & \| G (t, \cdot - x_0) - G (t, \cdot - y_0)
        \|_{3 / 2} + C \int_0^t \| \nabla G (t - s) \|_{3 / 2}  \| v_s \|_1 \mathd
        s\\
        & \leqslant & C \frac{| x_0 - y_0 |}{t^{5 / 6}} + C \int_0^t (t - s)^{- 5
        / 6}  \| v_s \|_1 \mathd s\\
        & \leqslant & C \frac{| x_0 - y_0 |}{t^{5 / 6}} + C | x_0 - y_0 | 
        \int_0^t (t - s)^{- 5 / 6}  (s^{- 1 / 2} + 1) \mathd s\\
        & \leqslant & C \frac{| x_0 - y_0 |}{t^{5 / 6}} + C | x_0 - y_0 |  \left(
        \int_0^{t / 2} (t - s)^{- 5 / 6} s^{- 1 / 2} \mathd s + \int_{t / 2}^t (t
        - s)^{- 5 / 6} s^{- 1 / 2} \mathd s \right)\\
        &  & + C | x_0 - y_0 | t^{1 / 6}\\
        & \leqslant & C (t^{- 5 / 6} + t^{- 1 / 3} + t^{1 / 6})  | x_0 - y_0 | =
        C \leqslant (t^{- 5 / 6} + 1)  | x_0 - y_0 | .
      \end{eqnarray*}
      \tmfoldedenv{The last two steps with $k = 3$, $p_3 = 3$ and $k = 4$, $p_4 =
      \infty$ are analogous.}{\paragraph{Step $k = 3$, $p_3 = 3$:}
      
      \begin{eqnarray*}
        \| v_t \|_3 & \leqslant & \| G (t, \cdot - x_0) - G (t, \cdot - y_0) \|_3
        + C \int_0^{t / 2} \| \nabla G (t - s) \|_3  \| v_s \|_1 \mathd s\\
        &  & + C \int_{t / 2}^t \| \nabla G (t - s) \|_{3 / 2}  \| v_s \|_{3 / 2}
        \mathd s\\
        & \leqslant & C \frac{| x_0 - y_0 |}{t^{7 / 6}} + C | x_0 - y_0 | 
        \int_0^{t / 2} (t - s)^{- 7 / 6}  (s^{- 1 / 2} + 1) \mathd s\\
        &  & + C | x_0 - y_0 |  \int_{t / 2}^t (t - s)^{- 5 / 6}  (s^{- 5 / 6} +
        1) \mathd s\\
        & \leqslant & C \frac{| x_0 - y_0 |}{t^{7 / 6}} + C | x_0 - y_0 | t^{- 7
        / 6}  (t^{1 / 2} + t)\\
        &  & + C | x_0 - y_0 |  (t^{- 5 / 6} + 1) t^{1 / 6}\\
        & \leqslant & C (t^{- 7 / 6} + t^{- 4 / 6} + t^{- 1 / 6} + t^{1 / 6})  |
        x_0 - y_0 | \leqslant C (t^{- 7 / 6} + 1)  | x_0 - y_0 | .
      \end{eqnarray*}
      
      \paragraph{Step $k = 4$, $p_4 = \infty$:}
      
      \begin{eqnarray*}
        \| v_t \|_{\infty} & \leqslant & \| G (t, \cdot - x_0) - G (t, \cdot -
        y_0) \|_{\infty} + C \int_0^{t / 2} \| \nabla G (t - s) \|_{\infty}  \|
        v_s \|_1 \mathd s\\
        &  & + C \int_{t / 2}^t \| \nabla G (t - s) \|_{3 / 2}  \| v_s \|_3
        \mathd s\\
        & \leqslant & C \frac{| x_0 - y_0 |}{t^{3 / 2}} + C | x_0 - y_0 | 
        \int_0^{t / 2} (t - s)^{- 3 / 2}  (s^{- 1 / 2} + 1) \mathd s\\
        &  & + C | x_0 - y_0 |  \int_{t / 2}^t (t - s)^{- 5 / 6}  (s^{- 7 / 6} +
        1) \mathd s\\
        & \leqslant & C \frac{| x_0 - y_0 |}{t^{3 / 2}} + C | x_0 - y_0 | t^{- 3
        / 2}  (t^{1 / 2} + t)\\
        &  & + C | x_0 - y_0 |  (t^{- 7 / 6} + 1) t^{1 / 6}\\
        & \leqslant & C (t^{- 3 / 2} + t^{- 1} + t^{- 1 / 2} + t^{1 / 6})  | x_0
        - y_0 | \leqslant C (t^{- 3 / 2} + 1)  | x_0 - y_0 | .
      \end{eqnarray*}
      \tmcolor{green}{Remark: if $\| u_{t, s}^x - u_{t, s}^y \|_p \leqslant C ((t
      - s)^{- a} + 1)  \| x - y \|_{\infty}$ holds for some $a > 0$, $p \in [1,
      \infty]$, then $\| u_{t, s}^x \|_p \leqslant C ((t - s)^{- a + 1 / 2} + 1)$
      will hold too.}}
    \end{proof}
    
    As a consequence we find the following estimate:
    
    \begin{corollary}
      \label{corollary-to-lemma}Let $f \in L^1 (\mathbbm{R}^2)$ and define $F :
      \mathbbm{R}^{2 N} \rightarrow \mathbbm{R}^{2 N}$ by $F_i (z) \assign
      \frac{1}{N}  \sum_{j \neq i} f (z_i - z_j)$, for $i = 1, \ldots, N$. Then,
      \[ | \mathbbm{E} (F (Z_{t, s}^{x, N})) -\mathbbm{E} (F (Z_{t, s}^{y, N}))
         |_{\infty} \leqslant C ((t - s)^{- 3 / 2} + 1)  \| f \|_1  | x - y
         |_{\infty} \]
      holds for $x, y \in \mathbbm{R}^{2 N}$, $t \in [0, T]$ and $Z_{t, s}^{x, N},
      Z_{t, s}^{y, N}$ given by (\ref{microscopic-mean-mixed}). 
    \end{corollary}
    
    Note that the interaction force $K^N$ is a function of this kind.
    
    \begin{proof}
      Let $i \in \{ 1, \ldots, N \}$.
      \[ \mathbbm{E} (F (Z^x_t))_i =\mathbbm{E} (F_i (Z^x_t)) = \frac{1}{N} 
         \sum_{j \neq i} \int f (z_i - z_j) u_t^{x_i} (z_i) u_t^{x_j} (z_j) \mathd
         z_i \mathd z_j . \]
      Therefore
      \begin{eqnarray*}
        | \mathbbm{E} (F (Z^x_t))_i -\mathbbm{E} (F (Z^y_t))_i | & = & \frac{1}{N}
        \left| \sum_{j \neq i} \int f (z_i - z_j)  (u_t^{x_i} (z_i) u_t^{x_j}
        (z_j) - u_t^{y_i} (z_i) u_t^{y_j} (z_j)) \mathd z_i \mathd z_j \right|\\
        & \leqslant & \frac{1}{N}  \sum_{j \neq i} \left| \int f (z_i - z_j)
        u_t^{x_i} (z_i)  (u_t^{x_j} (z_j) - u_t^{y_j} (z_j)) \mathd z_i \mathd z_j
        \right.\\
        &  & \qquad + \left. \int f (z_i - z_j) u_t^{y_j} (z_j)  (u_t^{x_i} (z_i)
        - u_t^{y_i} (z_i)) \mathd z_i \mathd z_j \right|\\
        & \leqslant & \frac{1}{N}  \sum_{j \neq i} (\| u_t^{x_j} - u_t^{y_j}
        \|_{\infty}  \| f \ast u_t^{x_i} \|_1 + \| u_t^{x_i} - u_t^{y_i}
        \|_{\infty}  \| f \ast u_t^{x_j} \|_1)\\
        & \leqslant & \frac{1}{N}  \sum_{j \neq i} C (t^{- 3 / 2} + 1)  | x - y
        |_{\infty}  (\| f \|_1  \| u_t^{x_i} \|_1 + \| f \|_1  \| u_t^{x_i}
        \|_1)\\
        & \leqslant & C (t^{- 3 / 2} + 1)  \| f \|_1  | x - y |_{\infty},
      \end{eqnarray*}
      by Lemma \ref{lemma}.
    \end{proof}
    
    We finally collect some standard estimates for the heat kernel which we
    required in the proof of Lemma \ref{lemma}.
    
    \begin{lemma}
      \label{p-norm-estimates-heat}{\tmem{($p$-norm estimates of the heat
      kernel)}} Let $G (t, x) \assign \frac{1}{2 \pi t} \mathe \tmop{xp} \left( -
      \frac{| x |^2}{2 t} \right)$ and $p \in [1, \infty]$. Then there exists a
      constant $C > 0$ such that the following holds:
      \begin{enumerateroman}
        \item \label{estimates-centered-heat-kernel}$\| G (t) \|_p \leqslant C
        \frac{1}{t^{1 - 1 / p}}$ and $\| \nabla_x G (t) \|_p \leqslant C
        \frac{1}{t^{3 / 2 - 1 / p}}$,
        
        \item \label{distance-shifted-heat-kernels}$\| G (t, \cdot - x_0) - G (t,
        \cdot - y_0) \|_p \leqslant C \frac{| x_0 - y_0 |}{t^{3 / 2 - 1 / p}}$.
      \end{enumerateroman}
    \end{lemma}
    
    \begin{proof}
      \ref{estimates-centered-heat-kernel}. $\| G (t) \|_p \leqslant C
      \frac{1}{t^{1 - 1 / p}}$ for $p \in [1, \infty]$.
      
      For $p = \infty$ the statement is clearly true.
      
      For $1 \leqslant p < \infty$
      \begin{eqnarray*}
        \| G (t) \|_p & = & \frac{1}{2 \pi t}  \left( \int \mathe \tmop{xp} \left(
        - \frac{p | x |^2}{2 t} \right) \mathd^2 x \right)^{1 / p}\\
        & = & \frac{C}{t^{1 - 1 / p}}  \left( \int \mathe \tmop{xp} (- p | y |^2)
        \mathd^2 y \right)^{1 / p}\\
        & \leqslant & \frac{C}{t^{1 - 1 / p}}  \left( \int \mathe \tmop{xp} (- |
        y |^2) \mathd^2 y \right)^{1 / p}\\
        & \leqslant & \frac{C}{t^{1 - 1 / p}} .
      \end{eqnarray*}
      Next we show that $\| \nabla_x G (t) \|_p \leqslant C \frac{1}{t^{3 / 2 - 1
      / p}}$ for $p \in [1, \infty]$. For $p = \infty$, since $a \exp (- a)$ is
      bounded, one has
      \[ | \nabla_x G (t, x) | = \left| \frac{x}{2 \pi t^2} \mathe \tmop{xp}
         \left( - \frac{| x |^2}{2 t} \right) \right| = \frac{C}{t^{3 / 2}} 
         \frac{| x |}{t^{1 / 2}} \mathe \tmop{xp} \left( - \frac{| x |^2}{2 t}
         \right) \leqslant \frac{C}{t^{3 / 2}}, \]
      for $(t, x) \in \mathbbm{R}_0^+ \times \mathbbm{R}^2$. For $1 \leqslant p <
      \infty$:
      \begin{eqnarray*}
        \| \nabla_x G (t) \|_p & = & \frac{1}{2 \pi t^2}  \left( \int | x |^p
        \mathe \tmop{xp} \left( - \frac{p | x |^2}{2 t} \right) \mathd^2 x
        \right)^{1 / p}\\
        & \leqslant & \frac{C}{t^{3 / 2 - 1 / p}}  \left( \int | y |^p \mathe
        \tmop{xp} (- p | y |^2) \mathd^2 y \right)^{1 / p}\\
        & \leqslant & \frac{C}{t^{3 / 2 - 1 / p}}  \left( \left\| | \cdot |^p
        \exp \left( - \frac{p | \cdot |^2}{2} \right) \right\|_{\infty} \right)^{1
        / p}  \left( \int \exp \left( - \frac{p | y |^2}{2} \right) \mathd^2 y
        \right)^{1 / p}\\
        & \leqslant & \frac{C}{t^{3 / 2 - 1 / p}}  \left\| | \cdot | \exp \left(
        - \frac{| \cdot |^2}{2} \right) \right\|_{\infty}  \left( \int \exp \left(
        - \frac{| y |^2}{2} \right) \mathd^2 y \right)\\
        & \leqslant & \frac{C}{t^{3 / 2 - 1 / p}} .
      \end{eqnarray*}
      \ref{distance-shifted-heat-kernels}. Let $V (t, x) \assign G (t, x - x_0) -
      G (t, x - y_0)$. For $p = \infty$, it follows from part \ref{estimates-centered-heat-kernel} that
      \[ | V (t, x) | \leqslant \| \nabla_x G (t) \|_{\infty}  | x_0 - y_0 |
         \leqslant C \frac{| x_0 - y_0 |}{t^{3 / 2}} . \]
      For $p = 1$ one can directly compute
      \[ \| V (t, \cdot) \|_1 \leqslant C \frac{| x_0 - y_0 |}{t^{1 / 2}} . \]
      For $1 < p < \infty$ then
      \begin{eqnarray*}
        \| V (t, \cdot) \|_p & \leqslant & \| V (t, \cdot) \|^{(p - 1) /
        p}_{\infty}  \| V (t, \cdot) \|^{1 / p}_1\\
        & \leqslant & C \left( \frac{| x_0 - y_0 |}{t^{3 / 2}} \right)^{(p - 1) /
        p}  \left( \frac{| x_0 - y_0 |}{t^{1 / 2}} \right)^{1 / p}\\
        & = & C \frac{| x_0 - y_0 |}{t^{3 / 2 - 1 / p}} .
      \end{eqnarray*}
    \end{proof}
    
    \section{Proof of the main theorem}\label{proof-main-theorem}
    
    In this section we prove Theorem \ref{main-result}, where we compare the
    regularised real trajectory $X^N$ given by (\ref{microscopic-eq-regularised})
    to the regularised mean-field trajectory $Y^N$ solving (\ref{microscopic-mean-regularised}) and show that both trajectories remain close with high
    probability if they start at the same point. This is done by two slightly
    different methods, depending on how big the elapsed time is. For large times
    we introduce the new process $Z^{N, X_s^N}_{t, s}$ starting at an intermediate
    time $s \in [0, t]$ and show it is close to $X_t^N$ and to $Y^N_t$. Recall
    that $Z^{N, X_s^N}_{t, s}$ is given by (\ref{microscopic-mean-mixed}) with
    initial condition $Z_{s, s}^N = X_s^N$. In order to simplify the notation we
    will omit the superindex in $Z^{N, X_s^N}_{t, s}$ refering to to the initial
    condition $X_s^N$ and denote just by $Z^N_{t, s}$ the solution of
    (\ref{microscopic-mean-mixed}) with initial condition $Z_{s, s}^N = X^N_s$. In
    particular, the identities $Z^N_{t, 0} = Y^N_t$ and $Z_{t, t}^N = X^N_t$ hold.
    Instead of directly considering the evolution of the difference $| X^N_t -
    Y^N_t |_{\infty}$ we work with a more complicated but technically convenient
    stochastic process, defined as follows: Let $T > 0$, $\alpha \in (0, 1 / 2)$
    and $\delta \assign \frac{1}{2}  \left( \frac{1}{2} - \alpha \right) > 0$. We
    consider the auxiliary process
    \begin{equation}
      J^N_t \assign \min \left\{ 1, \underset{0 \leqslant s \leqslant t}{\sup}
      \mathe^{C_N  (T - s)}  \underset{0 \leqslant \tau \leqslant s}{\sup}^{}
      (N^{\alpha} f_N (s - \tau)  | Z^N_{s, s} - Z^N_{s, \tau} |_{\infty} + N^{-
      \delta}) \right\}, \quad 0 \leqslant t \leqslant T, \label{J}
    \end{equation}
    \tmfoldedenv{where $C_N \assign 18 (\log N)^{3 / 4} $and $f_N (t) \assign \max
    \left\{ \frac{4}{t \log N + (\log N)^{- 1 / 4}}, 1
    \right\}$.}{\tmcolor{green}{Note: for $t \geqslant 0$ holds}
    
    \tmcolor{green}{\begin{tabular}{c}
      $t \log N \leqslant (\log N)^{- 1 / 4} \Rightarrow f (t) \geqslant \frac{f
      (0)}{2} = O ((\log N)^{1 / 4})$\\
      $t \log N \geqslant (\log N)^{- 1 / 4} \Rightarrow f (t) \in \left(
      \frac{2}{t \log N}, \frac{4}{t \log N} \right) \Rightarrow f = O ((t \log
      N)^{- 1})$
    \end{tabular}}}
    
    As we shall see the process $J^N_t$ helps us control the maximal distance $|
    Z^N_{s, s} - Z^N_{s, \tau} |_{\infty}$ for all intermediate times and the
    parameters in $J^N_t$ are optimised for the desired rate of convergence. We
    now explain how to express our problem in terms of this new process. For $s
    \geqslant \tau \geqslant 0$ let $a (\tau, s) \assign N^{\alpha} f_N (s - \tau)
    | Z^N_{s, s} - Z^N_{s, \tau} |_{\infty} + N^{- \delta}$. Since for each $t$
    the bound
    \[ N^{\alpha}  | X^N_t - Y^N_t |_{\infty} \leqslant \underset{0 \leqslant s
       \leqslant t}{\sup} \mathe^{C_N  (T - s)}  \underset{0 \leqslant \tau
       \leqslant s}{\sup}^{} a (\tau, s) \]
    holds true, $J^N_t < 1$ implies that$\underset{0 \leqslant s \leqslant
    t}{\sup} \mathe^{C_N  (T - s)}  \underset{0 \leqslant \tau \leqslant
    s}{\sup}^{} a (\tau, s) = J^N_t < 1$, and $| X^N_t - Y^N_t |_{\infty} < N^{-
    \alpha}$ follows. Moreover, since $\mathe^{C_N T}$ grows slower than
    $N^{\varepsilon}$ for any $\varepsilon > 0$, there exists $N_0 \in
    \mathbbm{N}$ depending on $T$ and $\alpha$ such that if $N \geqslant N_0$ then
    $J^N_0 = \mathe^{C_N T} N^{- \delta}$ is bounded by some constant, say $1 /
    2$. Therefore, we can estimate
    \begin{eqnarray*}
      \mathbbm{P} (\sup_{0 \leqslant t \leqslant T} | X^N_t - Y^N_t |_{\infty}
      \geqslant N^{- \alpha}) & \leqslant & \mathbbm{P} (J^N_t \geqslant 1)\\
      & \leqslant & \mathbbm{P} (J^N_t - J^N_0 \geqslant 1 / 2)\\
      & \leqslant & 2\mathbbm{E} (J^N_t - J^N_0)\\
      & = & 2 \int_0^t \mathbbm{E} (\partial^+_s J^N_s) \mathd s.
    \end{eqnarray*}
    The problem then reduces to finding a constant $C_{\gamma}$ for each $\gamma >
    0$ such that
    \[ \mathbbm{E} (\partial^+_t J^N_t) \leqslant C_{\gamma} N^{- \gamma} . \]
    In order to compute the right-derivative of $J^N_t$ we need the following
    lemma:
    
    \begin{lemma}
      \label{right-derivative}Let $g : [0, T] \times [0, T] \rightarrow
      \mathbbm{R}$ be a right-differentiable function and consider the function $f
      (t) \assign \sup_{0 \leqslant \tau \leqslant s \leqslant t} g (\tau, s)$ for
      $t \in [0, T]$. If the supremum of $g$ is not attained at any point of the
      diagonal $\{ (s, s) : s \in [0, T] \}$ then the right-derivative of $f$
      satisfies
      \[ \partial^+ f (t) \leqslant \max \{ 0, \partial^+_2 g (\tau, t) \}, \]
      for any $\tau \in [0, t]$ such that $(\tau, t)$ is maximal, meaning that $f
      (t) = g (\tau, t)$. Here the right-derivative $\partial_t^+$ for functions
      $\varphi$ in one variable is defined as
      \[ \partial_t^+ \varphi (t) \assign \lim_{h \rightarrow 0^+} \frac{\varphi
         (t + h) - \varphi (t)}{h} . \]
      For functions in several variables we denote by $\partial^+_i$ the partial
      right-derivative in the $i$-th variable.
    \end{lemma}
    
    \begin{proof}
      Let us denote by $(\tau_t, s_t)$ any maximal point of $g$ up to time $t$,
      i.e., any point such that $f (t) = g (\tau_t, s_t)$. We consider two cases.
      Assume first there exist $\tau_t, s_t$ satisfying the condition $0 \leqslant
      \tau_t \leqslant s_t < t$ such that $f (t) = g (\tau_t, s_t)$. In this
      situation it is clear (since $g$ is a right-continuous function) that $g
      (\tau_t, s_t)$ is also the supremum of $g$ over $0 \leqslant \tau \leqslant
      s \leqslant t + h$ for small enough $h > 0$. Therefore, $f (t + h) = f (t)$
      for $h$ in a small right-neighborhood of $0$ and so is $\partial^+ f (t) =
      0$.
      
      Next assume that the previous situation does not hold, that is, that the
      supremum of $g$ over $0 \leqslant \tau \leqslant s \leqslant t$ is only
      attained when $s = t$. In this case we also know that the first coordinate
      $\tau_t$ of any maximal point must satisfy $\tau_t < s_t = t$, since we
      assumed that the supremum is not attained on the diagonal. Using Lagrange
      multipliers one can easily deduce that the partial right-derivatives at any
      maximal point satisfy $\partial^+_1 g (\tau_t, t) = 0$ and $\partial^+_2 g
      (\tau_t, t) > 0$: The level curve through $(\tau_t, t)$ is tangent to the
      border of the triangle $\{ (\tau, s) \in \mathbbm{R}^2 : 0 \leqslant \tau
      \leqslant s \leqslant t \}$ where we are looking for the supremum. In this
      situation all maximal points $(\tau_t, t)$ lie on the horizontal line $s =
      t$ which is part of the triangle's border. This means that the
      right-gradient $(\partial^+_1, \partial^+_2)^t$ of $g$ at any such point
      $(\tau_t, t)$ is proportional to the vector $(0, 1)^t$, the outer normal to
      the triangle at $(\tau_t, t)$.
    \end{proof}
    
    Coming back to the computation of the right-derivative of $J^N_t$ (\ref{J}),
    note that we can write it as
    \[ J^N_t = \min \{ 1, \sup_{0 \leqslant \tau \leqslant s \leqslant t} g (\tau,
       s) \}, \]
    where
    \[ g (\tau, s) \assign \mathe^{C_N  (T - s)}  (N^{\alpha} f_N (s - \tau)  |
       Z^N_{s, s} - Z^N_{s, \tau} |_{\infty} + N^{- \delta}) . \]
    It is clear that $\partial^+_t J^N_t \leqslant \max \{ 0, \partial^+_t \sup_{0
    \leqslant \tau \leqslant s \leqslant t} g (\tau, s) \}$ holds. Moreover, the
    function $g$ satisfies the conditions of Lemma \ref{right-derivative} above,
    since the diagonal points are minimal for $g$ and therefore the supremum is
    not attained there. We can then apply the lemma to the function $\sup_{0
    \leqslant \tau \leqslant s \leqslant t} g (\tau, s)$ and find the following
    estimate, which holds for any maximal point $(\tau, t)$ of $g$, $0 \leqslant
    \tau \leqslant t$:
    \begin{eqnarray*}
      \partial^+_t J^N_t & \leqslant & \max \{ 0, - \mathe^{C_N (T - t)}  (C_N a
      (\tau, t) - N^{\alpha} f'_N (t - \tau)  | Z^N_{t, t} - Z^N_{t, \tau} |)
      \nobracket\\
      &  & \quad \qquad + \mathe^{C_N (T - t)} N^{\alpha} f_N (t - \tau) 
      \nobracket | K^N (Z^N_{t, t}) - \overline{K}^N_t (Z^N_{t, \tau}) | \}\\
      & = : & \max \{ 0, h (\tau, t) \} .
    \end{eqnarray*}
    Let us continue by trivially reducing the problem to a smaller set where $|
    Z^N_{s, s} - Z^N_{s, \tau} |_{\infty} \leqslant N^{- \alpha}$ holds for each
    $0 \leqslant \tau \leqslant s \leqslant t$. Consider the event $\mathcal{A}_t
    \assign \{ \partial^+_t J^N_t \geqslant 0 \}$. Since $\mathcal{A}_t \subseteq
    \{ h (\tau, t) \geqslant \partial^+_t J^N_t \}$ it holds that
    \begin{equation}
      \mathbbm{E} (\partial^+_t J^N_t) =\mathbbm{E} (\partial^+_t J^N_t |
      \mathcal{A}^c_t) +\mathbbm{E} (\partial^+_t J^N_t | \mathcal{A}_t) \leqslant
      0 +\mathbbm{E} (\partial^+_t J^N_t | \mathcal{A}_t) \leqslant \mathbbm{E} (h
      (\tau, t) |\mathcal{A}_t) . \label{trivial-bound}
    \end{equation}
    We shall prove that the latter is bounded by $C_{\gamma} N^{- \gamma}$ for
    some constant $C_{\gamma} \geqslant 0$. Note that in $\mathcal{A}_t$ one has
    $J^N_t \leqslant 1$ and in particular $\sup_{0 \leqslant \tau \leqslant s
    \leqslant t} | Z^N_{s, s} - Z^N_{s, \tau} |_{\infty} \leqslant N^{- \alpha}$
    holds. As a first estimate we can prove that in this set the bound $h (\tau,
    t)$ of the derivative $\partial^+_t J^N_t$ grows slower than $N^2$: Using that
    $| f'_N (t - \tau) | = C \log Nf^2_N (t - \tau) \leqslant C (\log N)^{3 / 2}$
    and $| K^N (Z^N_{t, t}) - \overline{K}^N_t (Z^N_{t, \tau}) | \leqslant
    CN^{\alpha}$ also hold, we find that in $\mathcal{A}_t$ is
    \begin{eqnarray}
      h (\tau, t) & \leqslant & \mathe^{C_N (T - t)}  (C_N a (\tau, t) +
      N^{\alpha}  | f'_N (t - \tau) |  | Z^N_{t, t} - Z^N_{t, \tau} |) \nonumber\\
      &  & + \mathe^{C_N (T - t)} N^{\alpha} f_N (t - \tau)  | K^N (Z^N_{t, t}) -
      \overline{K}^N_t (Z^N_{t, \tau}) | \nonumber\\
      & \leqslant & C \mathe^{C_N (T - t)}  ((\log N)^{3 / 4} + N^{\alpha}  (\log
      N)^{3 / 2} N^{- \alpha} + N^{\alpha}  (\log N)^{1 / 4} N^{\alpha})
      \nonumber\\
      & \leqslant & C \mathe^{C_N T} N^{3 / 2} < CN^2 .  \label{bound-in-A}
    \end{eqnarray}
    In order to prove $\mathbbm{E} (\partial^+_t J^N_t |\mathcal{A}_t) \leqslant
    C_{\gamma} N^{- \gamma}$ we distinguish between two cases depending on the
    difference $t - \tau$:
    
    \
    
    {\noindent}{\tmem{Case 1: $t - \tau \leqslant 2 (\log N)^{- 1}$.}}
    
    Here we show that $h (\tau, t) \leqslant 0$ holds outside a set of
    exponentially small measure and use that the regularised force $K^N$ is
    locally Lipschitz with constant of order $\log N$, which is a consequence of
    Lemma \ref{lipschitz-bound-G} and the law of large numbers (Proposition
    \ref{loln-general}): Note that in the notation of Lemma \ref{lipschitz-bound-G}, $K^N$ is equal to $K^{\nu (N)}$ for $\nu (N) \assign N^{\alpha}$ and so it
    is locally Lipschitz with bound $L^{\nu (N)}$, which was defined as
    \[ L^{\nu (N)}_i (y_1, \ldots, y_N) = - \frac{\chi}{N}  \sum_{j \neq i} l^{\nu
       (N)} (y_i - y_j) \]
    for
    \[ l^{\nu} (y) = \left\{ \begin{array}{ll}
         \frac{16}{| y |^2}, & | y | \geqslant 4 \nu^{- 1}\\
         \nu^2, & | y | \leqslant 4 \nu^{- 1}
       \end{array} \right. . \]
    Let us just write $L^N$ instead of $L^{\nu (N)}$ and denote by
    $\overline{L}_t^N$ the averaged version of $L^N$ given by
    \[ \overline{L}_{t, i}^N (y_1, \ldots, y_N) \assign - \chi (l^{\nu (N)} \ast
       \rho^N_t) (y_i) . \]
    Furthermore we consider the set
    \begin{equation}
      \mathcal{B}^1_t \assign \{ | K^N (Y^N_t) - \overline{K}^N_t (Y^N_t) |
      \leqslant N^{- (\alpha + \delta)} \} \cap \{ | L^N (Y^N_t) -
      \overline{L}^N_t (Y^N_t) | \leqslant C \} . \label{set-B1}
    \end{equation}
    In this event the real force $K^N$ acting on the i.i.d. particles $Y^N_t$ is
    well approximated by the mean-field force $\overline{K}^N_t$, which is
    globally Lipschitz. Moreover, the local Lipschitz constant $L^N$ of $K^N$ is
    of order $O (\log N)$ in $\mathcal{B}^1_t$. Indeed, since $l^{\nu (N)} =
    l_1^{\nu (N)} + l_{\infty}^{\nu (N)} \in L^1 (\mathbbm{R}^2) + L^{\infty}
    (\mathbbm{R}^2)$ with integrable part satisfying $\| l^{\nu (N)}_1 \|_1 = O
    (\log N)$ and $\rho_t^N$ is bounded in $L^1 (\mathbbm{R}^2) \cap L^{\infty}
    (\mathbbm{R}^2)$ uniformly in $N$ and $t \in [0, T]$, it holds that $\|
    \overline{L}_t^N \|_{\infty}$ is of order $O (\log N)$. Consequently the same
    estimate holds for $L^N$ in the set $\mathcal{B}^1_t$. Let us recall
    (\ref{trivial-bound}) and write
    \begin{equation}
      \mathbbm{E} (\partial^+_t J^N_t) \leqslant \mathbbm{E} (h (\tau, t) |
      \mathcal{A}_t) =\mathbbm{E} (h (\tau, t) |\mathcal{A}_t
      \backslash\mathcal{B}^1_t) +\mathbbm{E} (h (\tau, t) |\mathcal{A}_t \cap
      \mathcal{B}^1_t) . \label{split-with-B1}
    \end{equation}
    As a consequence of the law of large numbers (Proposition \ref{loln-general})
    the measure of the event $\Omega \backslash\mathcal{B}_t^1$ decays to zero as
    $N$ grows to infinity faster than any polynomial in $N$ (see Proposition
    \ref{measure-exceptional-sets} at the end of this section). Since $h (\tau,
    t)$ grows in the set $\mathcal{A}_t$ polynomially in $N$ only by estimate
    (\ref{bound-in-A}), we can find a positive constant $C_{\gamma}$ such that the
    first term in (\ref{split-with-B1}) satisfies
    \[ \mathbbm{E} (h (\tau, t) |\mathcal{A}_t \backslash\mathcal{B}^1_t)
       \leqslant C_{\gamma} N^{- \gamma} . \]
    It is therefore enough to prove that $h (\tau, t) \leqslant 0$ holds in
    $\mathcal{A}_t \cap \mathcal{B}^1_t$.
    
    Note that $h (\tau, t) \leqslant 0$ holds if for each $(\tau, t)$ where the
    supremum is attained the following inequality is true:
    \begin{eqnarray}
      f_N (t - \tau)  | K^N (Z^N_{t, t}) - \overline{K}^N_t (Z^N_{t, \tau}) | &
      \leqslant & - f'_N (t - \tau)  | Z^N_{t, t} - Z^N_{t, \tau} | \nonumber\\
      &  & + C_N  (f_N (t - \tau)  | Z^N_{t, t} - Z^N_{t, \tau} | + N^{- (\alpha
      + \delta)}) .  \label{iff-derivative-negative}
    \end{eqnarray}
    We next estimate the term $| K^N (Z^N_{t, t}) - \overline{K}^N_t (Z^N_{t,
    \tau}) |$ in the set in $\mathcal{A}_t \cap \mathcal{B}^1_t$ by splitting in
    three:
    \begin{eqnarray*}
      | K^N (Z^N_{t, t}) - \overline{K}^N_t (Z^N_{t, \tau}) | & \leqslant & | K^N
      (Z^N_{t, t}) - K^N (Z^N_{t, 0}) | + | K^N (Z^N_{t, 0}) - \overline{K}^N_t
      (Z^N_{t, 0}) |\\
      &  & + | \overline{K}^N_t (Z^N_{t, 0}) - \overline{K}^N_t (Z^N_{t, \tau}) |
      .
    \end{eqnarray*}
    The last term is the least problematic, since the function $\overline{K}^N_t$
    is globally Lipschitz. As noted before, the term in the middle is small in the
    event $\mathcal{B}^1_t$. For the first term we use that in this event the
    force $K^N$ is locally Lipschitz: we can apply Lemma \ref{lipschitz-bound-G}
    with $\nu (N) = N^{\alpha}$ and, since $| Z^N_{t, t} - Z^N_{t, 0} | \leqslant
    N^{- \alpha}$ in $\mathcal{A}_t$ and $| L^N (Y^N_t) - \overline{L}_t^N (Y^N_t)
    | \leqslant C$ in $\mathcal{B}^1_t$, we find
    \begin{eqnarray*}
      | K^N (Z^N_{t, t}) - K^N (Z^N_{t, 0}) | & \leqslant & 2 | L^N (Z^N_{t, 0}) |
      | Z^N_{t, t} - Z^N_{t, 0} | \leqslant 2 (C + \| \overline{L}_t^N
      \|_{\infty})  | Z^N_{t, t} - Z^N_{t, 0} |\\
      & \leqslant & 2 (C + \log N)  | Z^N_{t, t} - Z^N_{t, 0} | .
    \end{eqnarray*}
    Consequently,
    \begin{eqnarray*}
      | K^N (Z^N_{t, t}) - \overline{K}^N_t (Z^N_{t, \tau}) | & \leqslant & | K^N
      (Z^N_{t, t}) - K^N (Z^N_{t, 0}) | + | K^N (Z^N_{t, 0}) - \overline{K}^N_t
      (Z^N_{t, 0}) |\\
      &  & + | \overline{K}^N_t (Z^N_{t, 0}) - \overline{K}^N_t (Z^N_{t, \tau})
      |\\
      & \leqslant & 2 (C + \log N)  | Z^N_{t, t} - Z^N_{t, 0} | + N^{- (\alpha +
      \delta)} + L | Z^N_{t, t} - Z^N_{t, 0} |\\
      & \leqslant & (2 \log N + 2 C + L)  | Z^N_{t, t} - Z^N_{t, 0} | + L |
      Z^N_{t, t} - Z^N_{t, \tau} | + N^{- (\alpha + \delta)},
    \end{eqnarray*}
    where $L$ is the Lipschitz constant of $\overline{K}^N_t$ (uniform in $t \in
    [0, T]$). Now observe that, by the definition of $J^N_t$, $f_N (t - s)  |
    Z^N_{t, t} - Z^N_{t, s} | \leqslant f_N (t - \tau)  | Z^N_{t, t} - Z^N_{t,
    \tau} |$ holds for each $0 \leqslant s \leqslant t$. Therefore, we can choose
    a maybe greater $N_0$, depending now also on the Lipschitz constant $L$, such
    that for $N \geqslant N_0$ we find
    \begin{eqnarray*}
      | K^N (Z^N_{t, t}) - \overline{K}^N_t (Z^N_{t, \tau}) | & \leqslant & 2 (C +
      \log N)  \frac{f_N (t - \tau)}{f_N (t)}  | Z^N_{t, t} - Z^N_{t, \tau} | + L
      | Z^N_{t, t} - Z^N_{t, \tau} | + N^{- (\alpha + \delta)}\\
      & \leqslant & 3 \log Nf_N (t - \tau)  | Z^N_{t, t} - Z^N_{t, \tau} | + N^{-
      (\alpha + \delta)}\\
      & \leqslant & - \frac{f'_N (t - \tau)}{f_N (t - \tau)}  | Z^N_{t, t} -
      Z^N_{t, \tau} | + \frac{C_N}{f_N (t - \tau)} N^{- (\alpha + \delta)},
    \end{eqnarray*}
    which proves (\ref{iff-derivative-negative}). Here we used that $1 \leqslant f
    \leqslant C_N$ and $3 \log N (f_N (t - \tau))^2 \leqslant - f'_N (t - \tau)$.
    Consequently $h (\tau, t) \leqslant 0$ holds in the set $\mathcal{A}_t \cap
    \mathcal{B}^1_t$ and
    \[ \mathbbm{E} (\partial^+_t J^N_t) \leqslant \mathbbm{E} (h (\tau, t)
       |\mathcal{A}_t \backslash\mathcal{B}^1_t) +\mathbbm{E} (h (\tau, t)
       |\mathcal{A}_t \cap \mathcal{B}^1_t) \leqslant C_{\gamma} N^{- \gamma} \]
    as required.
    
    \
    
    {\noindent}{\tmem{Case 2: $t - \tau \geqslant 2 (\log N)^{- 1}$}}.
    
    The key now is to consider the process $Z^N_{t, s}$ starting at an appropiate
    intermediate time $s \in [0, t]$ and show that it is close to both the real
    trajectory $X_t^N$ and the mean-field trajectory $Y_t^N$. That it is close to
    the real trajectory is proven by the same argument as in the previous case,
    since the elapsed time $t - s$ is small enough. We compare $Z^N_{t, s}$ to the
    mean-field trajectory by an entirely different argument: we don't look at
    their trajectories but at their densities, which are close in $L^{\infty}$
    thanks to the diffusive effect of the Brownian Motion (Lemma 9 and Corollary
    10). We also need to split the interaction force $K^N$ into $K^N = K^N_1 +
    K^N_2$, where $K^N_2$ is the result of choosing a wider cutoff of order $(\log
    N)^{- 3 / 2}$ in the force kernel $k$ and $K^N_1 \assign K^N - K^N_2$. More
    precisely, let $k^N_2 \assign k^{\nu_2 (N)}$ for $\nu_2 (N) \assign (\log
    N)^{- 3 / 2}$ and define $k^N_1 \assign k^N - k^N_2$. The $i$-th components of
    $K^N_1$ and $K^N_2$ are then given by
    \begin{eqnarray}
      (K^N_1)_i (x_1, \ldots, x_N) & \assign & - \frac{\chi}{N}  \sum_{j \neq i}
      k_1^N (x_i - x_j) \nonumber\\
      \tmop{and} &  &  \nonumber\\
      (K^N_2)_i (x_1, \ldots, x_N) & \assign & - \frac{\chi}{N}  \sum_{j \neq i}
      k_2^N (x_i - x_j) .  \label{def-F1-F2}
    \end{eqnarray}
    We denote the local Lipschitz bound for $K^N_2$ given by Lemma \ref{lipschitz-bound-G} as $L^N_2 \assign L^{\nu_2 (N)}$ and its averaged version as
    $\overline{L}^N_{2, t}$, defined analogously to $\overline{L}^N_t$. Let us
    denote by $\mathcal{B}^2_t$ the intersection of the set $\mathcal{B}_t^1$ from
    the previous case and the set $\{ | L_2^N (Y^N_t) - \overline{L}_2^N (Y^N_t) |
    \leqslant C \}$ concerning the Lipschitz bound of the second part $K^N_2$ of
    $K^N$:
    \begin{equation}
      \mathcal{B}^2_t \assign \mathcal{B}_t^1 \cap \{ | L^N_2 (Y^N_t) -
      \overline{L}^N_{2, t} (Y^N_t) | \leqslant C \} . \label{set-B2}
    \end{equation}
    We write again
    \[ \mathbbm{E} (\partial^+_t J^N_t) \leqslant \mathbbm{E} (h (\tau, t) |
       \mathcal{A}_t) =\mathbbm{E} (h (\tau, t) |\mathcal{A}_t
       \backslash\mathcal{B}^2_t) +\mathbbm{E} (h (\tau, t) |\mathcal{A}_t \cap
       \mathcal{B}^2_t) . \]
    The first term is bounded as in the previous section: due to the exponential
    decay of the measure of $\mathcal{A}_t \backslash\mathcal{B}^2_t$ (proven in
    Proposition \ref{measure-exceptional-sets} below) in contrast to the milder
    polynomial growth of $h (\tau, t)$, we find a constant $C_{\gamma} \geqslant
    0$ such that
    \[ \mathbbm{E} (h (\tau, t) |\mathcal{A}_t \backslash\mathcal{B}^2_t)
       \leqslant C_{\gamma} N^{- \gamma} . \]
    It remains to show that also $\mathbbm{E} (h (\tau, t) |\mathcal{A}_t \cap
    \mathcal{B}^2_t) \leqslant C_{\gamma} N^{- \gamma}$ holds (for a possibly
    different constant $C_{\gamma}$, which we don't rename for simplicity of
    notation).
    
    Notice that $\mathbbm{E} (h (\tau, t) |\mathcal{A}_t \cap \mathcal{B}^2_t)
    \leqslant C_{\gamma} N^{- \gamma}$ holds if the following inequality is true:
    \begin{eqnarray}
      f_N (t - \tau) \mathbbm{E} (| K^N (Z^N_{t, t}) - \overline{K}^N_t (Z^N_{t,
      \tau}) | | \mathcal{A}_t \cap \mathcal{B}^2_t) & \leqslant & - f'_N (t -
      \tau) \mathbbm{E} (| Z^N_{t, t} - Z^N_{t, \tau} | | \mathcal{A}_t \cap
      \mathcal{B}^2_t) \nonumber\\
      &  & + C_N f_N (t - \tau) \mathbbm{E} (| Z^N_{t, t} - Z^N_{t, \tau} | |
      \mathcal{A}_t \cap \mathcal{B}^2_t) \nonumber\\
      &  & + C_N N^{- (\alpha + \delta)} \mathbbm{P} (\mathcal{A}_t \cap
      \mathcal{B}^2_t) \nonumber\\
      &  & + C_{\gamma} N^{- \gamma} . \label{theorem-master-inequality} 
    \end{eqnarray}
    To this end we write as before
    \begin{eqnarray}
      | K^N (Z^N_{t, t}) - \overline{K}^N_t (Z^N_{t, \tau}) | & \leqslant & | K^N
      (Z^N_{t, t}) - K^N (Z^N_{t, 0}) | + | K^N (Z^N_{t, 0}) - \overline{K}^N_t
      (Z^N_{t, 0}) | \nonumber\\
      &  & + | \overline{K}^N_t (Z^N_{t, 0}) - \overline{K}^N_t (Z^N_{t, \tau}) |
      .  \label{first-split}
    \end{eqnarray}
    The last two terms can be bounded in the same way as in the previous section,
    but for $| K^N (Z^N_{t, t}) - K^N (Z^N_{t, 0}) |$ we can no longer use the
    corresponding Lipschitz bound from Lemma \ref{lipschitz-bound-G} directly.
    Here we need to add the intermediate time $s = t - (\log N)^{- 3 / 2}$ and to
    split the force into $K^N = K^N_1 + K^N_2$ as described in (\ref{def-F1-F2}),
    which results in
    \begin{eqnarray}
      | K^N (Z^N_{t, t}) - K^N (Z^N_{t, 0}) | & \leqslant & | K^N (Z^N_{t, t}) -
      K^N (Z^N_{t, s}) | + | K^N (Z^N_{t, s}) - K^N (Z^N_{t, 0}) | \nonumber\\
      & \leqslant & | K^N (Z^N_{t, t}) - K^N (Z^N_{t, s}) | + | K^N_1 (Z^N_{t,
      s}) - K^N_1 (Z^N_{t, 0}) | \nonumber\\
      &  & + | K^N_2 (Z^N_{t, s}) - K^N_2 (Z^N_{t, 0}) | .  \label{split-the-force}
    \end{eqnarray}
    We can now use the Lipschitz bound for the first and third terms in
    (\ref{split-the-force}): In $\mathcal{A}_t \cap \mathcal{B}^2_t$ it holds that
    \begin{eqnarray}
      | K^N (Z^N_{t, t}) - K^N (Z^N_{t, s}) | & \leqslant & 2 | L^N (Z^N_{t, s}) |
      | Z^N_{t, t} - Z^N_{t, s} | \nonumber\\
      & \leqslant & 6 (C + \| \overline{L}_t^N \|_{\infty})  | Z^N_{t, t} -
      Z^N_{t, s} | \nonumber\\
      & \leqslant & 12 \log N | Z^N_{t, t} - Z^N_{t, s} | \nonumber\\
      & \leqslant & 12 \log N \frac{f_N (t - \tau)}{f_N (t - s)}  | Z^N_{t, t} -
      Z^N_{t, \tau} | \nonumber\\
      & \leqslant & 6 (\log N)^{3 / 4} f_N (t - \tau)  | Z^N_{t, t} - Z^N_{t,
      \tau} |,  \label{split-1}
    \end{eqnarray}
    since $f_N (s - r)  | Z_{s, s} - Z_{s, r} | \leqslant f_N (t - \tau)  |
    Z^N_{t, t} - Z^N_{t, \tau} |$ is true for each $0 \leqslant r \leqslant s
    \leqslant t$ and also $f_N (t - s) \geqslant 2 (\log N)^{1 / 4}$. We
    analogously obtain the following estimate for the third term in (\ref{split-the-force})
    \begin{eqnarray}
      | K^N_2 (Z^N_{t, s}) - K^N_2 (Z^N_{t, 0}) | & \leqslant & 2 | L_2^N (Z^N_{t,
      0}) |  | Z^N_{t, s} - Z^N_{t, 0} | \nonumber\\
      & \leqslant & 2 (\| \overline{L}_{2, t}^N \|_{\infty} + C)  | Z^N_{t, s} -
      Z^N_{t, 0} | \nonumber\\
      & \leqslant & 4 \log \log N | Z^N_{t, s} - Z^N_{t, 0} | \nonumber\\
      & \leqslant & 4 \log \log Nf_N (t - \tau)  \left( \frac{1}{f_N (t - s)} +
      \frac{1}{f_N (t)} \right)  | Z^N_{t, t} - Z^N_{t, \tau} | \nonumber\\
      & \leqslant & 8 \log \log Nf_N (t - \tau)  | Z^N_{t, t} - Z^N_{t, \tau} | .
      \label{split-3}
    \end{eqnarray}
    The estimate provided by the local Lipschitz bound from Lemma \ref{lipschitz-bound-G} works for $| K^N (Z^N_{t, t}) - K^N (Z^N_{t, s}) |$ and $| K^N_2
    (Z^N_{t, s}) - K^N_2 (Z^N_{t, 0}) |$ because in the first term the elapsed
    time $t - s$ is small enough (so we can compensate the $\log N$ order coming
    from the derivative of $K^N$ with $(f_N (t - s))^{- 1}$) and in the other one
    the force $K^N_2$ has a milder derivative which is of order $\log \log N$
    only. For the remaining term $| K^N_1 (Z^N_{t, s}) - K^N_1 (Z^N_{t, 0}) |$ in
    (\ref{split-the-force}) we use that the probability densities of $Z^N_{t, s}$
    and $Z^N_{t, 0}$ are close in $L^{\infty}$ by Lemma \ref{lemma} and its
    Corollary \ref{corollary-to-lemma}. Note that in order to complete the last
    argument we need independence of the particles and, although the mean-field
    particles $Z^{1, N}_{t, 0}, \ldots, Z^{N, N}_{t, 0}$ are pairwise independent,
    this does not hold for the particles $Z^{1, N}_{t, s}, \ldots, Z^{N, N}_{t,
    s}$ (recall that by definition $Z^N_{t, s} = Z^{X^N_s}_{t, s}$ and that
    $Z^N_{t, 0} = Z^{Y^N_s}_{t, s}$ for $t \geqslant s$). For this reason, instead
    of considering the processes starting at time $s$ at the r.v. $X^N_s$ and
    $Y^N_s$ respectively, it is convenient to first fix the starting points at
    time $s$ to be some given points $x, y \in \mathbbm{R}^{2 N}$ and to compare
    the corresponding (product distributed) processes $Z^{x, N}_{t, s}$ and $Z^{y,
    N}_{t, s}$. This being done, we can recover the original processes $Z^N_{t,
    s}$ and $Z^N_{t, 0}$ by writting $\mathbbm{E} (| K^N_1 (Z^N_{t, s}) - K^N_1
    (Z^N_{t, 0}) | |\mathcal{A}_t \cap \mathcal{B}^2_t)$ as
    \begin{equation}
      \int_{(x, y) \in (Z^N_{s, s}, Z^N_{s, 0}) (\mathcal{A}_t \cap
      \mathcal{B}^2_t)} \mathbbm{E} (| K^N_1 (Z^{x, N}_{t, s}) - K^N_1 (Z^{y,
      N}_{t, s}) | |\mathcal{A}_t \cap \mathcal{B}^2_t) \mathbbm{P} (X^N_s \in
      \mathd x, Y^N_s \in \mathd y) . \label{integral-delta}
    \end{equation}
    Let us then fix $x, y \in \mathbbm{R}^{2 N}$ and write
    \begin{eqnarray*}
      \mathbbm{E} (| K^N_1 (Z^{x, N}_{t, s}) - K^N_1 (Z^{y, N}_{t, s}) |
      |\mathcal{A}_t \cap \mathcal{B}^2_t) & = & \mathbbm{E} (| K^N_1 (Z^{x,
      N}_{t, s}) - K^N_1 (Z^{y, N}_{t, s}) | | (\mathcal{A}_t \cap
      \mathcal{B}^2_t) \backslash\mathcal{C}_t^{x, y})\\
      &  & +\mathbbm{E} (| K^N_1 (Z^{x, N}_{t, s}) - K^N_1 (Z^{y, N}_{t, s}) |
      |\mathcal{A}_t \cap \mathcal{B}^2_t \cap \mathcal{C}_t^{x, y}),
    \end{eqnarray*}
    where we introduced the new set
    \begin{eqnarray}
      \mathcal{C}^{x, y}_t & \assign & \{ | K^N_1 (Z^{x, N}_{t, s}) -\mathbbm{E}
      (K^N_1 (Z^{x, N}_{t, s})) | \leqslant N^{- (\alpha + \delta)} \} \label{set-C} \nonumber\\
      &  & \cap \{ | K^N_1 (Z^{y, N}_{t, s}) -\mathbbm{E} (K^N_1 (Z^{y, N}_{t,
      s})) | \leqslant N^{- (\alpha + \delta)} \}, 
    \end{eqnarray}
    for $s = t - (\log N)^{- 3 / 2}$. By Proposition \ref{measure-exceptional-sets} below the measure of the set $\Omega \backslash\mathcal{C}_t^{x, y}$ is
    exponentially small. Also note that the bound given in Proposition
    \ref{measure-exceptional-sets} does not depend of the points $x, y$. Since
    $K^N_1$ is of order $O (N^{\alpha})$ we can find a constant $C_{\gamma} > 0$
    such that
    \[ \mathbbm{E} (| K^N_1 (Z^{x, N}_{t, s}) - K^N_1 (Z^{y, N}_{t, s}) | |
       (\mathcal{A}_t \cap \mathcal{B}^2_t) \backslash\mathcal{C}_t^{x, y})
       \leqslant C_{\gamma} N^{- \gamma} . \]
    Next we estimate $| K^N_1 (Z^{x, N}_{t, s}) - K^N_1 (Z^{y, N}_{t, s}) |$ in
    the set $\mathcal{A}_t \cap \mathcal{B}^2_t \cap \mathcal{C}_t^{x, y}$. We
    write
    \begin{eqnarray*}
      | K^N_1 (Z^{x, N}_{t, s}) - K^N_1 (Z^{y, N}_{t, s}) | & \leqslant & | K^N_1
      (Z^{x, N}_{t, s}) -\mathbbm{E} (K^N_1 (Z^{x, N}_{t, s})) | + | K^N_1 (Z^{y,
      N}_{t, s}) -\mathbbm{E} (K^N_1 (Z^{y, N}_{t, s})) |\\
      &  & + | \mathbbm{E} (K^N_1 (Z^{x, N}_{t, s})) -\mathbbm{E} (K^N_1 (Z^{y,
      N}_{t, s})) | .
    \end{eqnarray*}
    In $\mathcal{C}^{x, y}_t$ the first two terms are bounded. For the remaining
    term $| \mathbbm{E} (K^N_1 (Z^{x, N}_{t, s})) -\mathbbm{E} (K^N_1 (Z^{y,
    N}_{t, s})) |$ we use the following fact: both processes $Z^{x, N}_{t, s}$ and
    $Z^{y, N}_{t, s}$ evolved according to the mean-field dynamics during a period
    of time $t - s$, which is long enough to ensure that the densities of $Z^{x,
    N}_{t, s}$ and $Z^{y, N}_{t, s}$ are close if their starting positions $x$ and
    $y$ are close. It follows that the difference $| \mathbbm{E} (K^N_1 (Z^{x,
    N}_{t, s})) -\mathbbm{E} (K^N_1 (Z^{y, N}_{t, s})) |$ is also small in that
    case (Corollary \ref{corollary-to-lemma}). More precisely,
    \begin{eqnarray*}
      | K^N_1 (Z^{x, N}_{t, s}) - K^N_1 (Z^{y, N}_{t, s}) | & \leqslant & | K^N_1
      (Z^{x, N}_{t, s}) -\mathbbm{E} (K^N_1 (Z^{x, N}_{t, s})) | + | K^N_1 (Z^{y,
      N}_{t, s}) -\mathbbm{E} (K^N_1 (Z^{y, N}_{t, s})) |\\
      &  & + | \mathbbm{E} (K^N_1 (Z^{x, N}_{t, s})) -\mathbbm{E} (K^N_1 (Z^{y,
      N}_{t, s})) |\\
      & \leqslant & 2 N^{- (\alpha + \delta)} + \frac{| x - y |}{(t - s)^{3 / 2}}
      \| k_1^N \|_1,
    \end{eqnarray*}
    is true in the event $\mathcal{A}_t \cap \mathcal{B}^2_t \cap
    \mathcal{C}_t^{x, y}$. Consequently the expected value in $\mathcal{A}_t \cap
    \mathcal{B}^2_t$ for fixed starting points $x, y$ can be bounded as:
    \[ \mathbbm{E} (| K^N_1 (Z^{x, N}_{t, s}) - K^N_1 (Z^{y, N}_{t, s}) |
       |\mathcal{A}_t \cap \mathcal{B}^2_t) \leqslant \frac{| x - y |}{(t - s)^{3
       / 2}}  \| k_1^N \|_1 + 2 N^{- (\alpha + \delta)} \mathbbm{P} (\mathcal{A}_t
       \cap \mathcal{B}^2_t) + C_{\gamma} N^{- \gamma} . \]
    Next, with (\ref{integral-delta}) we find an estimate for the original
    processes
    \begin{eqnarray*}
      \mathbbm{E} (| K^N_1 (Z^N_{t, s}) - K^N_1 (Z^N_{t, 0}) | |\mathcal{A}_t \cap
      \mathcal{B}^2_t) & \leqslant & \frac{\mathbbm{E} (| Z^N_{s, s} - Z^N_{s, 0}
      | | \mathcal{A}_t \cap \mathcal{B}^2_t)}{(t - s)^{3 / 2}}  \| k_1^N \|_1\\
      &  & + 2 N^{- (\alpha + \delta)} \mathbbm{P} (\mathcal{A}_t \cap
      \mathcal{B}^2_t)\\
      &  & + C_{\gamma} N^{- \gamma}\\
      & \leqslant & (\log N)^{3 / 4} \mathbbm{E} (| Z^N_{s, s} - Z^N_{s, 0} | |
      \mathcal{A}_t \cap \mathcal{B}^2_t)\\
      &  & + 2 N^{- (\alpha + \delta)} \mathbbm{P} (\mathcal{A}_t \cap
      \mathcal{B}^2_t)\\
      &  & + C_{\gamma} N^{- \gamma},
    \end{eqnarray*}
    where for the last inequality we used that $t - s = (\log N)^{- 3 / 2}$ and
    $\| k_1^N \|_1 \leqslant (\log N)^{- 3 / 2}$. Consequently,
    \begin{eqnarray*}
      \mathbbm{E} (| K^N_1 (Z^N_{t, s}) - K^N_1 (Z^N_{t, 0}) | |\mathcal{A}_t \cap
      \mathcal{B}^2_t) & \leqslant & (\log N)^{3 / 4}  \frac{f_N (t - \tau)}{f_N
      (s)} \mathbbm{E} (| Z^N_{t, t} - Z^N_{t, \tau} | | \mathcal{A}_t \cap
      \mathcal{B}^2_t)\\
      &  & + 2 N^{- (\alpha + \delta)} \mathbbm{P} (\mathcal{A}_t \cap
      \mathcal{B}^2_t) + C_{\gamma} N^{- \gamma}\\
      & \leqslant & (\log N)^{3 / 4} f_N (t - \tau) \mathbbm{E} (| Z^N_{t, t} -
      Z^N_{t, \tau} | | \mathcal{A}_t \cap \mathcal{B}^2_t)\\
      &  & + 2 N^{- (\alpha + \delta)} \mathbbm{P} (\mathcal{A}_t \cap
      \mathcal{B}^2_t) + C_{\gamma} N^{- \gamma} .
    \end{eqnarray*}
    Together with (\ref{split-1}) and (\ref{split-3}) this covers all three terms
    appearing in (\ref{split-the-force}). We can adapt $N_0 \in \mathbbm{N}$
    chosen at the beggining of the proof so that for $N \geqslant N_0$:
    \begin{eqnarray*}
      \mathbbm{E} (| K^N (Z^N_{t, t}) - K^N (Z^N_{t, 0}) | | \mathcal{A}_t \cap
      \mathcal{B}^2_t) & \leqslant & 7 (\log N)^{3 / 4} f_N (t - \tau) \mathbbm{E}
      (| Z^N_{t, t} - Z^N_{t, \tau} | | \mathcal{A}_t \cap \mathcal{B}^2_t)\\
      &  & + 8 \log \log Nf_N (t - \tau) \mathbbm{E} (| Z^N_{t, t} - Z^N_{t,
      \tau} | | \mathcal{A}_t \cap \mathcal{B}^2_t)\\
      &  & + 2 N^{- (\alpha + \delta)} \mathbbm{P} (\mathcal{A}_t \cap
      \mathcal{B}^2_t) + C_{\gamma} N^{- \gamma}\\
      & \leqslant & 8 (\log N)^{3 / 4} f_N (t - \tau) \mathbbm{E} (| Z^N_{t, t} -
      Z^N_{t, \tau} | | \mathcal{A}_t \cap \mathcal{B}^2_t)\\
      &  & + 2 N^{- (\alpha + \delta)} \mathbbm{P} (\mathcal{A}_t \cap
      \mathcal{B}^2_t) + C_{\gamma} N^{- \gamma} .
    \end{eqnarray*}
    Going back to (\ref{first-split}) we use this last estimate for the first
    term, the bound
    \[ | K^N (Z^N_{t, 0}) - \overline{K}^N_t (Z^N_{t, 0}) | \leqslant N^{- (\alpha
       + \delta)} \]
    in $\mathcal{A}_t \cap \mathcal{B}^2_t$ for the second term and the Lipschitz
    continuity of $\overline{K}^N_t$
    \[ | \overline{K}^N_t (Z^N_{t, 0}) - \overline{K}^N_t (Z^N_{t, \tau}) |
       \leqslant L | Z^N_{t, 0} - Z^N_{t, \tau} | \leqslant L \left( 1 + \frac{f_N
       (t - \tau)}{f_N (t)} \right)  | Z^N_{t, t} - Z^N_{t, \tau} | \]
    for the third. Bringing everything together, (\ref{first-split}) becomes
    \begin{eqnarray*}
      \mathbbm{E} (| K^N (Z^N_{t, t}) - \overline{K}^N_t (Z^N_{t, \tau}) | |
      \mathcal{A}_t \cap \mathcal{B}^2_t) & \leqslant & 8 (\log N)^{3 / 4} f_N (t
      - \tau) \mathbbm{E} (| Z^N_{t, t} - Z^N_{t, \tau} | | \mathcal{A}_t \cap
      \mathcal{B}^2_t)\\
      &  & + 2 N^{- (\alpha + \delta)} \mathbbm{P} (\mathcal{A}_t \cap
      \mathcal{B}^2_t) + C_{\gamma} N^{- \gamma}\\
      &  & + N^{- (\alpha + \delta)} \mathbbm{P} (\mathcal{A}_t \cap
      \mathcal{B}^2_t)\\
      &  & + L (1 + f_N (t - \tau)) \mathbbm{E} (| Z^N_{t, t} - Z^N_{t, \tau} | |
      \mathcal{A}_t \cap \mathcal{B}^2_t)\\
      & \leqslant & 9 (\log N)^{3 / 4} f_N (t - \tau) \mathbbm{E} (| Z^N_{t, t} -
      Z^N_{t, \tau} | | \mathcal{A}_t \cap \mathcal{B}^2_t)\\
      &  & + 3 N^{- (\alpha + \delta)} \mathbbm{P} (\mathcal{A}_t \cap
      \mathcal{B}^2_t) + C_{\gamma} N^{- \gamma},
    \end{eqnarray*}
    which is true if $N$ is greater than some new $N_0$ depending now also on the
    Lipschitz constant $L$. Finally, from $f_N (t - \tau) \leqslant 2$ it follows
    that
    \begin{eqnarray*}
      \mathbbm{E} (| K^N (Z^N_{t, t}) - \overline{K}^N_t (Z^N_{t, \tau}) | |
      \mathcal{A}_t \cap \mathcal{B}^2_t) & \leqslant & 18 (\log N)^{3 / 4}
      \mathbbm{E} (| Z^N_{t, t} - Z^N_{t, \tau} | | \mathcal{A}_t \cap
      \mathcal{B}^2_t)\\
      &  & + 3 N^{- (\alpha + \delta)} \mathbbm{P} (\mathcal{A}_t \cap
      \mathcal{B}^2_t) + C_{\gamma} N^{- \gamma},
    \end{eqnarray*}
    proving (\ref{theorem-master-inequality}). As a consequence:
    \[ \mathbbm{E} (\partial^+_t J^N_t) \leqslant \mathbbm{E} (h (\tau, t)
       |\mathcal{A}_t \backslash\mathcal{B}^2_t) +\mathbbm{E} (h (\tau, t)
       |\mathcal{A}_t \cap \mathcal{B}^2_t) \leqslant 2 C_{\gamma} N^{- \gamma} =
       : \tilde{C}_{\gamma} N^{- \gamma} . \]

    It just remains to \ estimate the measure of the complementary sets of
    $\mathcal{B}_t^1, \mathcal{B}_t^2$ and $\mathcal{C}_t^{x, y}$ as defined in
    (\ref{set-B1}), (\ref{set-B2}) and (\ref{set-C}). The constants $T > 0$,
    $\alpha \in (0, 1 / 2)$ and $\delta > 0$ are the ones we fixed at the
    beginning of this section.
    
    \begin{proposition}
      \label{measure-exceptional-sets}{\tmem{(Measure of the exceptional sets)}}
      For each $\gamma > 0$ there exists a positive constant $C_{\gamma}$ such
      that
      \begin{enumerateroman}
        \item \label{loln-0}$\mathbbm{P} (S^1_t \cup S^2_t \cup S^3_t) \leqslant
        C_{\gamma} N^{- \gamma}$ for each $0 \leqslant t \leqslant T$ , where
        
        \begin{tabular}{cc}
          $S^1_t \assign \{ | K^N (Y^N_t) - \overline{K}^N_t (Y^N_t) |_{\infty}
          \geqslant N^{- (\alpha + \delta)} \}$, & \\
          $S^2_t \assign \{ | L^N (Y^N_t) - \overline{L}^N_t (Y^N_t) |_{\infty}
          \geqslant 1 \}$, & $S^3_t \assign \{ | L^N_2 (Y^N_t) -
          \overline{L}^N_{2, t} (Y^N_t) |_{\infty} \geqslant 1 \}$.
        \end{tabular}
        
        {\noindent}Consequently $\mathbbm{P} (\Omega \backslash\mathcal{B}^1_t)
        \leqslant C_{\gamma} N^{- \gamma}$ and $\mathbbm{P} (\Omega
        \backslash\mathcal{B}^2_t) \leqslant C_{\gamma} N^{- \gamma}$ hold for
        each $0 \leqslant t \leqslant T$.
        
        \item \label{loln-s} $\mathbbm{P} (| K^N_1 (Z^{x, N}_{t, s})
        -\mathbbm{E} (K^N_1 (Z^{x, N}_{t, s})) |_{\infty} \geqslant N^{- (\alpha +
        \delta)}) \leqslant C_{\gamma} N^{- \gamma}$ holds for any $x \in
        \mathbbm{R}^{2 N}$ and any $T \geqslant t \geqslant s \geqslant 0$
        satisfying $t - s \geqslant (\log N)^{- r}$ for some $r \geqslant 0$.
        Consequently, $\mathbbm{P} (\Omega \backslash\mathcal{C}_t^{x, y})
        \leqslant 2 C_{\gamma} N^{- \gamma}$ for any $x, y \in \mathbbm{R}^{2 N}$
        and $0 \leqslant s \leqslant t \leqslant T$.
      \end{enumerateroman}
    \end{proposition}
    
    \begin{proof}
      \ref{loln-0}. First note that the mean-field force $\overline{K}^N_{t, i}
      (Y^N_t)$ can be written in terms of the expected value of $K^N$ as
      $\overline{K}^N_{t, i} (Y^N_t) =\mathbbm{E}_{(- i)} (K^N_i (Y^N_t))$ and
      therefore the first set $S^1_t$ is equal to the set $\{ \sup_{1 \leqslant i
      \leqslant N} | K^N_i (Y^N_t) -\mathbbm{E}_{(- i)} (K^N_i (Y^N_t)) |
      \geqslant N^{- (\alpha + \delta)} \}$. Moreover, $Y^1_t, \ldots, Y^N_t$ are
      already pairwise independent and the $L^{\infty}$-norm of its probability
      density $\rho^N_t$ is bounded uniformly in $N$ and $t \in [0, T]$ by
      Proposition \ref{Linf-estimates}. Therefore, from Proposition \ref{loln-general} follows the existence of a constant $C_{\gamma} > 0$, independent
      of $t$, with
      \[ \mathbbm{P} (S^1_t) =\mathbbm{P} (| K^N (Y^N_t) - \overline{K}^N_t
         (Y^N_t) |_{\infty} \geqslant N^{- (\alpha + \delta)}) \leqslant
         C_{\gamma} N^{- \gamma}, \]
      for each $t \in [0, T]$.
      
      The remaining sets $S^2_t$ and $S^3_t$ can be expressed in terms of the
      expected value of $L^N$ resp. $L_2^N$ in an analogous way. Also note that
      both $| N^{- \alpha} L^N_i (x) |$ and $| N^{- \alpha} L^N_{2 i} (x) |$ are
      bounded by $C \chi \min \{ N^{\alpha}, | x |^{- 1} \}$. Proposition
      \ref{loln-general} then implies for $S^2_t$ that
      \begin{eqnarray*}
        \mathbbm{P} (| L^N (Y^N_t) - \overline{L}^N_t (Y^N_t) |_{\infty} \geqslant
        1) & = & \mathbbm{P} (N^{- \alpha}  | L^N (Y^N_t) - \overline{L}^N_t
        (Y^N_t) |_{\infty} \geqslant N^{- \alpha})\\
        & \leqslant & \mathbbm{P} (N^{- \alpha}  | L^N (Y^N_t) - \overline{L}^N_t
        (Y^N_t) |_{\infty} \geqslant N^{- (\alpha + \delta)})\\
        & \leqslant & C_{\gamma} N^{- \gamma},
      \end{eqnarray*}
      and in the same manner that $\mathbbm{P} (S^3_t) =\mathbbm{P} (| L_2 (Y^N_t)
      - \overline{L}^N_{2, t} (Y^N_t) |_{\infty} \geqslant 1) \leqslant C_{\gamma}
      N^{- \gamma}$ for each $t \in [0, T]$.
      
      \ref{loln-s}. Let $T \geqslant t \geqslant s \geqslant 0$ be such that $t -
      s \geqslant (\log N)^{- r}$ holds for some $r \geqslant 0$. First notice
      that for each fixed starting point $x \in \mathbbm{R}^{2 N}$ the processes
      $Z^{x, 1, N}_{t, s}, \ldots, Z^{x, N, N}_{t, s}$ are pairwise independent.
      Furthermore, the probability density $u_{t, s}^{x, i, N}$ of $Z^{x, i,
      N}_{t, s}$ satisfies
      \[ \| u_{t, s}^{x, i, N} \|_{\infty} \leqslant C ((t - s)^{- 1} + 1)
         \leqslant C (\log N)^r \]
      for $i = 1, \ldots, N$, by Lemma \ref{lemma}, meaning that the growth of $\|
      u_{t, s}^{x, i, N} \|_{\infty}$ is only logarithmic in $N$ and consequently
      condition (\ref{loln-density-condition}) is fulfilled independently of the
      times $t, s$ and the exponent $r$. Therefore there exists a constant
      $C_{\gamma} > 0$ such that, for any such $t, s$:
      \[ \mathbbm{P} (| K^N_1 (Z^{x, N}_{t, s}) -\mathbbm{E} (K^N_1 (Z^{x, N}_{t,
         s})) |_{\infty} \geqslant N^{- (\alpha + \delta)}) \leqslant C_{\gamma}
         N^{- \gamma} . \]
    \end{proof}
    
    \section{Proofs of Propositions \ref{Linf-estimates} and \ref{hˆlder-estimates}}\label{proofs-of-estimates}
    
    We now give the proof of Proposition \ref{Linf-estimates}.
    
    \begin{proof}
      One first proves the boundedness of $\rho$ in $L^p$ for each $1 < p <
      \infty$. The $L^{\infty}$ estimate follows from this fact and the
      boundedness of $\nabla c = k \ast \rho$ by an iterative argument.
      
      \
      
      {\noindent}{\tmem{Step 1: Uniform bounds in $L^p$, $p < \infty$.}}
      
      Notice that under the assumptions $\rho_0 \in L^p (\mathbbm{R}^2)$ for each
      $p \in [1, \infty]$. Then $\rho \in L^{\infty} (0, T ; L^p (\mathbbm{R}^2))$
      for any $T > 0$ and $1 \leqslant p < \infty$. See either {\cite[Proposition
      17]{blanchet_two-dimensional_2006}} or {\cite[Lemma
      2.7]{fernandez_uniqueness_2016}} for a proof.
      
      \tmfoldedenv{\ }{Note for step 2: In {\cite[Lemma
      4.1]{kowalczyk_preventing_2005}} they also need to show $L^p$ estimates
      but rely instead on the following equi-integrability property
      \[ \sup_{t \in [0, T]} \int (\rho - m)_+ \xrightarrow[m \rightarrow
         \infty]{} 0, \]
      which might not hold in our setting since their Hypothesis H5.1 used to
      prove it in {\cite[Theorem 5.3]{calvez_volume_2006}} does not hold: $h
      (0^+) = \log (0^+) \ngtr - \infty$. But no equi-integrability is needed with
      our extra assumption $\rho_0 \in L^{\infty}$, which implies that $\rho_0 \in
      L^p$ for each $p$.}
      
      {\noindent}{\tmem{Step 2: Uniform bounds in $L^{\infty}$.}}
      
      For this step we follow {\cite[Lemma 3.2]{calvez_volume_2006}} and
      {\cite[Lemma 4.1]{kowalczyk_preventing_2005}}. The second reference is
      much more detailed but only handles bounded domains. The proof can
      nevertheless be adapted for the whole space $\mathbbm{R}^2$ as described in
      the first paper.
      
      The following computations are performed only formally. One can justify them
      by performing the proof for the solutions $\rho^N$ of the regularised
      equation (\ref{macroscopic-eq-regularised}) and then passing to the limit.
      
      Let $\rho_m \assign (\rho - m)_+$. First notice that $\nabla c_{} = \| k
      \ast \rho \|_{\infty}$ is uniformly bounded: $\| k \ast \rho \|_{\infty}
      \leqslant C (\| \rho \|_3 + \| \rho \|_1)$ since $k \in k^{3 / 2} +
      L^{\infty}$, and the right hand side is uniformly bounded by the first step.
      We then prove the inequality
      \begin{eqnarray}
        \frac{\mathd}{\mathd t} \int \rho_m^p \mathd x & \leqslant & - Cp^2  \|
        \nabla c \|^2_{\infty}  \int \rho_m^p \mathd x \nonumber\\
        &  & + C^2 p^4  \| \nabla c \|^4_{\infty}  \left( \int \rho_m^{p / 2}
        \mathd x \right)^2 + Cp^2  \| \nabla c \|^2_{\infty} .  \label{L-inf-estimates}
      \end{eqnarray}
      From this we will conclude that $\sup_{t \in [0, T]}  \| \rho_m \|_p$ is
      bounded independently of $p$. The proof is then complete after taking the
      limit $p \rightarrow \infty$.
      
      We first multiply on both sides of the Keller Segel equation
      (\ref{macroscopic-eq}) by $\rho_m^{p - 1}$ and integrate to find
      \[ \frac{1}{p}  \frac{\mathd}{\mathd t} \int \rho^p_m \mathd x = \int \nabla
         \cdot \left( \nabla \rho + \chi \nabla c \rho \right) \rho^{p - 1}_m . \]
      Let $\Omega_t \assign \{ \rho (t) \geqslant m \}$ and notice that $\Omega_t$
      is uniformly bounded: $1 = \| \rho (t) \|_1 \geqslant m | \Omega_t |$. Then
      the integral on the right hand side equals
      \begin{eqnarray*}
        \int_{\Omega_t} \nabla \cdot \left( \nabla \rho + \chi  (k \ast \rho) \rho
        \right) \rho^{p - 1}_m & = & - \int_{\Omega_t} \left( \nabla \rho + \chi 
        (k \ast \rho) \rho \right) \nabla \rho^{p - 1}_m\\
        & = & - (p - 1)  \int \rho_m^{p - 2}  | \nabla \rho_m |^2 + \chi  (p - 1)
        \int \rho \rho_m^{p - 2} \nabla c \cdot \nabla \rho_m\\
        & = & - (p - 1)  \int \rho_m^{p - 2}  | \nabla \rho_m |^2 + \chi  (p - 1)
        \int \rho_m^{p - 1} \nabla c \cdot \nabla \rho_m\\
        &  & + \chi m (p - 1)  \int \rho_m^{p - 2} \nabla c \cdot \nabla \rho_m .
      \end{eqnarray*}
      Using that $\rho_m^{(p - k) / 2} \nabla \rho_m^{p / 2} = \frac{p}{2}
      \rho_m^{p - (k / 2 + 1)} \nabla \rho_m$ for any $k \in \mathbbm{R}$ the last
      expression equals
      \[ - \frac{4 (p - 1)}{p^2}  \int | \nabla \rho^{p / 2}_m |^2 + \frac{2 \chi 
         (p - 1)}{p}  \int \rho_m^{p / 2} \nabla c \cdot \nabla \rho^{p / 2}_m +
         \frac{2 \chi m (p - 1)}{p}  \int \rho_m^{(p - 2) / 2} \nabla c \cdot
         \nabla \rho^{p / 2}_m . \]
      For the last two terms we use the following Young's inequality: $| a \cdot b
      | \leqslant \frac{1}{4}  | a |^2 + | b |^2$ for any two vectors $a, b \in
      \mathbbm{R}^2$. Hence
      \[ (p - 1)  \int \chi \rho_m^{p / 2} \nabla c \cdot \frac{2}{p} \nabla
         \rho^{p / 2}_m \leqslant \frac{(p - 1)}{p^2} \int | \nabla \rho^{p / 2}_m
         |^2 + \chi^2  (p - 1)  \| \nabla c \|_{\infty}^2 \int \rho_m^p \]
      and
      \begin{eqnarray*}
        (p - 1)  \int \chi m \rho_m^{(p - 2) / 2} \nabla c \cdot \frac{2}{p}
        \nabla \rho^{p / 2}_m & \leqslant & \frac{(p - 1)}{p^2}  \int | \nabla
        \rho^{p / 2}_m |^2 + \chi^2 m^2  (p - 1)  \| \nabla c \|_{\infty}^2
        \int_{\Omega_t} \rho_m^{p - 2}\\
        & \leqslant & \frac{(p - 1)}{p^2}  \int | \nabla \rho^{p / 2}_m |^2 + C
        (p - 1)  \| \nabla c \|_{\infty}^2 \int_{\Omega_t} (\rho_m^p + 1)\\
        & \leqslant & \frac{(p - 1)}{p^2}  \int | \nabla \rho^{p / 2}_m |^2 + C
        (p - 1)  \| \nabla c \|_{\infty}^2 \int \rho_m^p\\
        &  & + C (p - 1)  \| \nabla c \|^2_{\infty}  | \Omega_t | .
      \end{eqnarray*}
      All together
      \begin{eqnarray*}
        \frac{\mathd}{\mathd t} \int \rho^p_m \mathd x & \leqslant & - \frac{2 (p
        - 1)}{p}  \int | \nabla \rho^{p / 2}_m |^2 + Cp (p - 1)  \int \rho^p_m +
        Cp (p - 1) \\
        & \leqslant & - \int | \nabla \rho^{p / 2}_m |^2 + Cp^2 \int \rho^p_m +
        Cp^2,
      \end{eqnarray*}
      for $p$ big enough.
      
      Now we use the Galiardo-Nirenberg-Sobolev inequality followed by Young's
      inequality
      \[ \| u \|^2_2 \leqslant C_{\tmop{GNS}}  \| \nabla u \|_2  \| u \|_1
         \leqslant \frac{C_{\tmop{GNS}}^2}{4}  \| u \|^2_1 + \| \nabla u \|^2_2 \]
      for $u = \rho_m^{p / 2}$:
      \[ (C + 1) p^2 \int \rho^p_m \leqslant \frac{(C + 1)^2 C^2_{\tmop{GNS}}
         p^4}{4}  \left( \int \rho^{p / 2}_m \right)^2 + \int | \nabla \rho^{p /
         2}_m |^2 . \]
      Therefore
      \[ \frac{\mathd}{\mathd t} \int \rho^p_m \mathd x \leqslant - p^2  \int
         \rho^p_m + Cp^4  \left( \int \rho^{p / 2}_m \right)^2 + Cp^2, \]
      which proves (\ref{L-inf-estimates}).
      
      Let now $w_j = \int \rho_m^{2^j}$, $S_j \assign \sup_{t \in [0, T]} \int
      \rho^{2^j}_m$ for $j \in \mathbbm{N}$. Then
      \[ \frac{\mathd}{\mathd t} w_j \mathd x \leqslant - 2^{2 j} w_j + 2^{2 j} 
         (C 2^{2 j} S^2_{j - 1} + C) . \]
      The solution of
      \[ \frac{\mathd}{\mathd t} v = - \varepsilon v + \varepsilon C \]
      is $v (t) = \mathe^{- \varepsilon t} v_0 + C (1 - \mathe^{- \varepsilon
      t})$. If we set $v_0 = w_j (0)$ it holds
      \[ w_j \leqslant v \leqslant w_j (0) + C 2^{2 j} S^2_{j - 1} + C \leqslant
         \| \rho_0 \|^{2^j}_{\infty}  | \Omega_0 | + C 2^{2 j} S^2_{j - 1} + C. \]
      It follows that
      \[ S_j = \sup_{t \in [0, T]} w_j \leqslant C \max \{ \| \rho_0
         \|^{2^j}_{\infty}, 2^{2 j} S^2_{j - 1} + 1 \} . \]
      For $\tilde{S}_j \assign S_j  \| \rho_0 \|^{2^{- j}}_{\infty}$ holds the
      following:
      \[ \tilde{S}_j \leqslant C \max \{ 1, 2^{2 j}  \tilde{S}^2_{j - 1} \} . \]
      Hence
      \begin{eqnarray*}
        \log_+  \tilde{S}_j & \leqslant & \max \{ \log_+ C, \log_+ C 2^{2 j} 
        \tilde{S}^2_{j - 1} \}\\
        & \leqslant & 2 \log_+ \tilde{S}_{j - 1} + j \log 4 + C,
      \end{eqnarray*}
      which implies $2^{- j} \log_+  \tilde{S}_j - 2^{- (j - 1)} \log_+ 
      \tilde{S}_{j - 1} \leqslant j 2^{- j} \log 4 + C 2^{- j}$ for $j \in
      \mathbbm{N}$. Adding up both sides over $j = 1, \ldots, J$ we find
      \begin{eqnarray*}
        2^{- J} \log_+  \tilde{S}_J - \log_+  \tilde{S}_0 & = & \sum_{j = 1}^J
        2^{- j} \log_+  \tilde{S}_j - 2^{- (j - 1)} \log_+  \tilde{S}_{j - 1}\\
        & \leqslant & \sum_{j = 1}^{\infty} j 2^{- j} \log 4 + C 2^{- j}
        \leqslant C,
      \end{eqnarray*}
      for a constant $C$ independent of $J$. Since $\tilde{S}_0 \leqslant \sup_{t
      \in [0, T]} \frac{\| \rho (t) \|_1}{\| \rho_0 \|_{\infty}}$ is also bounded,
      we conclude that $S^{2^{- j}}_j = \left( \sup_{t \in [0, T]} \int
      \rho^{2^j}_m \right)^{2^{- j}} = \sup_{t \in [0, T]} \left( \int
      \rho^{2^j}_m \right)^{2^{- j}} \leqslant C$ for some contant $C$ not
      depending on $j$. We finally perform the limit $j \rightarrow \infty$ and
      conclude
      \[ \sup_{t \in [0, T]}  \| \rho_m \|_{\infty} = \sup_{t \in [0, T]} \lim_{j
         \rightarrow \infty} \| \rho_m \|_{2^j} \leqslant \lim_{j \rightarrow
         \infty} \sup_{t \in [0, T]} \| \rho_m \|_{2^j} \leqslant C. \]
      \tmfoldedenv{\ }{Note that this also proves that $\sup_{t \in [0, T]} \|
      \rho_m \|_p \leqslant C$ for each $p$ and not only for those of the form
      $2^j$, since for $p \leqslant 2^j$ it holds $\| \rho_m \|_p \leqslant C \|
      \rho_m \|_{2^j}  | \Omega_t | \leqslant C \| \rho_m \|_{2^j}$. But this
      bound is not uniform in $t$ and $p$ if we don't introduce a lower cutoff $m
      > 0$.}
    \end{proof}
    
    \tmfoldedenv{We finish with the proof of Proposition \ref{hˆlder-estimates}.}{Proof of the first statement, following {\cite[Lemma
    2.8]{fernandez_uniqueness_2016}}:
    
    The key is to prove that
    \[ \partial_t \rho, \nabla \rho, \partial_t \rho^N, \nabla \rho^N \in L^p ((0,
       T) \times \mathbbm{R}^2) \]
    for any $p \in (1, \infty)$ uniformly in $N$ and to use Morrey's inequality
    \[ \| u \|_{C^{0, \alpha} (\mathbbm{R}^n)} \leqslant C \| u \|_{W^{1, p}
       (\mathbbm{R}^n)}, \quad \text{where } n < p \leqslant \infty \text{ and }
       \alpha = 1 - \frac{n}{p} \]
    for $n = 3$ and some $p > 3$.
    
    Observe that
    \begin{equation}
      \partial_t \rho - \Delta \rho = \chi \rho^2 + \chi (k \ast \rho) \cdot
      \nabla \rho = : f \label{heat-equation}
    \end{equation}
    and
    \begin{equation}
      \partial_t \rho^N - \Delta \rho^N = \chi \rho \nabla \cdot (k \ast \rho) +
      \chi (k^N \ast \rho^N) \cdot \nabla \rho^N = : f^N . \label{heat-equation-N}
    \end{equation}
    We shall use regularity results for the heat equation. According to
    {\cite[Lemma 2.7]{fernandez_uniqueness_2016}}, since $\rho_0 \in L^p
    (\mathbbm{R}^2)$ for any $p \in [2, \infty)$, for any weak solution $\rho$ and
    any $T < \infty$ it holds $\rho \in L^{\infty} (0, T ; L^p (\mathbbm{R}^2))$,
    and $\nabla \rho \in L^2 ((0, T) \times \mathbbm{R}^2)$. Since $k \ast \rho
    \in L^{\infty} ((0, T) \times \mathbbm{R}^2)$, we conclude that $f \in L^2
    ((0, T) \times \mathbbm{R}^2)$. Next we show that the same holds true for
    $f^N$. By {\cite[Lemma 13]{blanchet_two-dimensional_2006}}, $\nabla
    \sqrt{\rho^N} \in L^2 ((0, T) \times \mathbbm{R}^2)$ with uniform bounds in
    $N$. From this we deduce that $\rho^N$ is in $L^{\infty} (0, T ; L^p
    (\mathbbm{R}^2))$ for any $p \in (1, \infty)$ and that $\nabla \rho^N \in L^2
    ((0, T) \times \mathbbm{R}^2)$ holds. Indeed, by the
    Galiardo-Nirenberg-Sobolev inequality for $q > 2$
    \[ \int | \rho^N |^{q / 2} \leqslant (C_{\tmop{GNS}}^{(q)})^{q / 2}  \left(
       \int \left| \nabla \sqrt{\rho^N} \right|^2 \right)^{q / 2 - 2}, \]
    so $\rho^N \in L^{\infty} (0, T ; L^p (\mathbbm{R}^2))$ for any $p \in (1,
    \infty)$. Moreover,
    \[ \int | \nabla \rho^N |^2 = \int \rho^N  \left| \frac{1}{\sqrt{\rho^N}}
       \nabla \rho^N \right|^2 = 4 \int \rho^N  \left| \nabla \sqrt{\rho^N}
       \right|^2 \leqslant C \| \rho^N \|_{\infty}  \left\| \nabla \sqrt{\rho^N}
       \right\|^2_2 \]
    and so is $\nabla \rho^N \in L^2 ((0, T) \times \mathbbm{R}^2)$. As before, it
    holds $k^N \ast \rho^N \in L^{\infty} ((0, T) \times \mathbbm{R}^2)$. We
    finally show that $\nabla \cdot (k^N \ast \rho^N) \in L^2 ((0, T) \times
    \mathbbm{R}^2)$ and with this we can conclude that the right hand side $f^N$
    of (\ref{heat-equation-N}) is in $L^2 ((0, T) \times \mathbbm{R}^2)$. Note
    that $\| \nabla \cdot k^N \|_1 = 1$ (actually, $\nabla \cdot k^N
    \rightharpoonup \delta$ in $\mathcal{D}'$). Therefore, by Young's inequality,
    \[ \| (\nabla \cdot k^N) \ast \rho^N \|_2 \leqslant \| \nabla \cdot k^N \|_1 
       \| \rho^N \|_2 \leqslant C \]
    uniformly in $N$.
    
    By assumption is $\rho_0 \in H^1 (\mathbbm{R}^2)$, so the maximal regularity
    result of the heat equation in $L^2$-spaces for (\ref{heat-equation}) and
    (\ref{heat-equation-N}) implies that
    \[ \rho, \rho^N \in L^2 (0, T ; H^2 (\mathbbm{R}^2)) \cap L^{\infty} (0, T ;
       H^1 (\mathbbm{R}^2)), \quad \forall N \in \mathbbm{N}, \]
    with all these norms bounded by some constant depending only on $\rho_0$ and
    the final time $T$. Notice that the the norms in $L^2 ((0, T) \times
    \mathbbm{R}^2)$ of $f$ and $f^N$ depend only on $\rho_0$ and $T$. It follows
    that $\nabla \rho \in L^p ((0, T) \times \mathbbm{R}^2)$ for any $p \in [2,
    \infty)$ and consequently $(k \ast \rho) \nabla \rho \in L^p ((0, T) \times
    \mathbbm{R}^2)$ (see Remark \ref{embeddings} below).
    
    Now since $\rho_0 \in H^2 (\mathbbm{R}^2) \subseteq W^{2 - 2 / p, p}
    (\mathbbm{R}^2)$ for any $2 < p \leqslant 4$ (see Remark \ref{embeddings}
    again), from the maximal regularity result of the heat equation in
    $L^p$-spaces it follows that $\partial_t \rho, \partial^2_{i j} \rho,
    \partial_t \rho^N, \partial^2_{i j} \rho^N \in L^p ((0, T) \times
    \mathbbm{R}^2)$ for $i, j = 1, 2$, $N \in \mathbbm{N}$, and in particular we
    find
    \[ \partial_t \rho, \nabla \rho, \partial_t \rho^N, \nabla \rho^N \in L^p ((0,
       T) \times \mathbbm{R}^2), \quad \text{for each } N \in \mathbbm{N} \text{
       and } p \in (2, 4] . \]
    Morrey's inequality
    \[ \| u \|_{C^{0, \alpha} (\mathbbm{R}^n)} \leqslant C \| u \|_{W^{1, p}
       (\mathbbm{R}^n)}, \quad \text{where } n < p \leqslant \infty \text{ and }
       \alpha = 1 - \frac{n}{p}, \]
    with $n = 3$ and $p \in (2, 4]$ implies that $\rho, \rho^N \in C^{0, \alpha}
    ((0, T) \times \mathbbm{R}^2)$ for each $N \in \mathbbm{N}$ and $0 < \alpha
    \leqslant 1 / 4$.}
    
    \begin{proof}
      \ref{hˆlder}. From the proof of {\cite[Lemma
      2.8]{fernandez_uniqueness_2016}} follows that $\rho$ and $\rho^N$ are in
      $C^{0, \alpha} ((0, T) \times \mathbbm{R}^2)$ for each $N \in \mathbbm{N}$
      and $0 < \alpha \leqslant 1 / 4$. For fixed positive $\alpha \leqslant 1 /
      4$ this means that for each $t, s > 0$ and $x, y \in \mathbbm{R}^2$
      \[ | \rho (t, x) - \rho (s, y) | \leqslant C (| x - y |^{\alpha} + | t - s
         |^{\alpha}) \]
      holds for some positive constant $C$, which depends on the norms of $\rho$,
      $\partial_t \rho$ and $\nabla \rho$ in $L^p \left( \left( 0, {\nobreak} T
      \right) \times {\nobreak} \mathbbm{R}^2 \right)$ with $p \assign \frac{3}{1
      - \alpha} \in (3, 4]$. Since $\rho_0 \in H^2 (\mathbbm{R}^2) \subseteq C^{0,
      \alpha} (\mathbbm{R}^2)$, by taking $t = s$ above we find
      \[ | \rho (t, x) - \rho (t, y) | \leqslant C \max \{ \| \rho \|_p + \|
         \partial_t \rho \|_p + \| \nabla \rho \|_p, [\rho_0]_{0, \alpha} \}  | x
         - y |^{\alpha}, \]
      for each $t \in [0, T]$ and $x, y \in \mathbbm{R}^2$, and the analogous
      inequality holds for $\rho^N$. Let
      \[ C_1 \geqslant C \max \{ \| \rho \|_{W^{1, p} ((0, T) \times
         \mathbbm{R}^2)}, \sup_{N \in \mathbbm{N}} \| \rho^N \|_{W^{1, p} ((0, T)
         \times \mathbbm{R}^2)}, [\rho_0]_{0, \alpha} \} . \]
      Then $\rho, \rho^N \in L^{\infty} (0, T ; C^{0, \alpha} (\mathbbm{R}^2))$
      and $[\rho (t)]_{0, \alpha}, [\rho^N (t)]_{0, \alpha} \leqslant C_1$ for
      each $N \in \mathbbm{N}$ and $t \in [0, T]$.
      
      \ref{lipschitz}. Let $w = \phi \ast \rho = - \log | \cdot | \ast \rho$. We
      need to prove that $- \nabla w^N = k^N \ast \rho^N$ and $- \nabla w = k \ast
      \rho$ are Lipschitz continuous in $\mathbbm{R}^2$ uniformly in $N \in
      \mathbbm{N}$ and $t \in [0, T]$. It is then enough to show that all second
      derivatives of $w^N$ and $w$ are uniformly bounded. More precisely, we find
      \[ \| \partial_{i j} w^N (t) \|_{\infty} \leqslant C (\| \rho^N (t) \|_1 +
         \| \rho^N (t) \|_{\infty} + [\rho^N (t)]_{0, \alpha}), \quad N \in
         \mathbbm{N} \]
      and
      \[ \| \partial_{i j} w (t) \|_{\infty} \leqslant C (\| \rho (t) \|_1 + \|
         \rho (t) \|_{\infty} + [\rho (t)]_{0, \alpha}) \]
      for some constant $C > 0$. These are uniformly bounded on $[0, T]$ and $N
      \in \mathbbm{N}$ by part \ref{hˆlder} and Proposition \ref{Linf-estimates}.
      We just write down the proof for the limiting case $k \ast \rho$. For $k^N
      \ast \rho^N$ the steps are completely analogous.
      
      We split the integral as follows:
      \[ \partial_{i j} w (t, x) = \int_{| x - y | \leqslant 1} \partial_{i j}
         \phi (x - y) \rho (y) \mathd y + \int_{| x - y | \geqslant 1} \partial_{i
         j} \phi (x - y) \rho (y) \mathd y. \]
      The second term, since $\partial_{i j} \phi (x - y) \leqslant \frac{C}{| x -
      y |^2}$, is bounded by $C \| \rho \|_1$. For the first term we write
      \begin{eqnarray*}
        \int_{| x - y | \leqslant 1} \partial_{i j} \phi (x - y) \rho (y) \mathd y
        & = & \int_{| x - y | \leqslant 1} \partial_{i j} \phi (x - y)  (\rho (y)
        - \rho (x)) \mathd y\\
        &  & + \rho (x)  \int_{| x - y | \leqslant 1} \partial_{i j} \phi (x - y)
        \mathd y\\
        & = & \int_{| x - y | \leqslant 1} \partial_{i j} \phi (x - y)  (\rho (y)
        - \rho (x)) \mathd y\\
        &  & - \rho (x)  \int_{| x - y | = 1} \partial_i \phi (x - y) \nu_j (y)
        \mathd S (y) .
      \end{eqnarray*}
      Therefore in absolute value
      \begin{eqnarray*}
        \left| \int_{| x - y | \leqslant 1} \partial_{i j} \phi (x - y) \rho (y)
        \mathd y \right| & \leqslant & C [\rho (t)]_{0, \alpha}  \int_{| x - y |
        \leqslant 1} \frac{1}{| x - y |^{2 - \alpha}} \mathd y\\
        &  & + C \| \rho (t) \|_{\infty} .
      \end{eqnarray*}
      Putting all together we find
      \[ \| \partial_{i j} w (t) \|_{\infty} \leqslant C (\| \rho (t) \|_1 + \|
         \rho (t) \|_{\infty} + [\rho (t)]_{0, \alpha}) . \]
    \end{proof}
    
    \begin{remark}
      \label{embeddings}Below we list the space embeddings we used in the proof.
      \begin{enumerateroman}
        \item  $H^2 (\mathbbm{R}^2) \subseteq C^{0, \alpha} (\mathbbm{R}^2)$, for
        any $0 < \alpha < 1$, by the Sobolev embedding theorem {\cite[Theorem
        2.31]{demengel_functional_2012}}.
        
        \item If $f \in L^2 (0, T ; H^2 (\mathbbm{R}^2)) \cap L^{\infty} (0, T ;
        H^1 (\mathbbm{R}^2))$ then $\nabla_x f \in L^p ((0, T) \times
        \mathbbm{R}^2)$ for any $p \in (1, 4)$.
        
        Since $W^{1, 2} (\mathbbm{R}^2) \subseteq L^q (\mathbbm{R}^2)$ for any $2
        \leqslant q < \infty$, we have that
        \[ \nabla_x f \in L^2 (0, T ; L^q (\mathbbm{R}^2)) \cap L^{\infty} (0, T ;
           L^2 (\mathbbm{R}^2)), \quad \text{for } 2 \leqslant q < \infty . \]
        We then use the interpolation inequality
        \[ \| u \|_{p_{\theta}} \leqslant \| u \|_{p_0}^{\theta}  \| u \|_{p_1}^{1
           - \theta}, \quad \text{for } \theta \in [0, 1], \quad
           \frac{1}{p_{\theta}} = \frac{\theta}{p_0} + \frac{1 - \theta}{p_1} \]
        and find
        \begin{eqnarray*}
          \int_0^T \| \nabla_x f (t) \|_p^p \mathd t & \leqslant & \int_0^T (\|
          \nabla_x f \|^{\theta}_2  \| \nabla_x f \|^{1 - \theta}_{p_1})^p\\
          & \leqslant & \sup_{t \in [0, T]} \| \nabla_x f (t) \|_2^{\theta p} 
          \int_0^T \| \nabla_x f \|^{(1 - \theta) p}_{p_1} .
        \end{eqnarray*}
        By choosing $\theta = 1 / 2$ it holds for $p < 4$ that $p (1 - \theta)
        \leqslant 2$ and $p_1 = \frac{2 p (1 - \theta)}{2 - \theta p} = \frac{2
        p}{4 - p} < \infty$, and so is the right hand side of the last inequality
        finite.
        
        \item  $H^2 (\mathbbm{R}^2) \subseteq W^{2 - 2 / p, p} (\mathbbm{R}^2)$,
        for any $2 < p \leqslant 4$.
        
        By the Sobolev embedding theorem for fractional spaces {\cite[Theorem
        7.58]{adams_sobolev_1975}} it holds
        \[ H^2 (\mathbbm{R}^2) \subseteq W^{1 + 2 / p, p} (\mathbbm{R}^2), \quad
           \text{for any } 2 < p < \infty . \]
        Since $1 + 2 / p \geqslant 2 - 2 / p$ holds if $p \leqslant 4$, we
        conclude that
        \[ H^2 (\mathbbm{R}^2) \subseteq W^{1 + 2 / p, p} (\mathbbm{R}^2)
           \subseteq W^{2 - 2 / p, p} (\mathbbm{R}^2), \quad \text{for any } 2 < p
           \leqslant 4. \]
      \end{enumerateroman}
    \end{remark}
    
    \section*{Acknowledgements}
    
    We are greatful to Martin Kolb, Detlef D{\"u}rr, Christian Stinner and Tomasz
    Cieslak for many valuable discussions on this topic and to Miguel de Benito
    Delgado for reading and improving the manuscript. We further acknowledge the
    finantial support of the Studienstiftung des Deutschen Volkes.

    \end{document}